\documentclass[12pt]{amsart}
\usepackage[latin1]{inputenc}
\usepackage{color}
\usepackage{bbm}
\usepackage{amsmath,amssymb}
\usepackage{enumerate}

\newtheorem{thm}{Theorem}[section]
\newtheorem{cor}[thm]{Corollary}
\newtheorem{lem}[thm]{Lemma}
\newtheorem{prop}[thm]{Proposition}
\theoremstyle{definition}
\newtheorem{defin}[thm]{Definition}

\numberwithin{equation}{section}

\makeatletter
\newcommand{\lin}{\operatorname{span}}
\newcommand{\supp}{\operatorname{supp}}
\newcommand{\dist}{\operatorname{dist}}

\newcommand{\dif}{\,\mathrm{d}}

\newcommand{\cT}{\mathcal{T}}

\newcommand{\cM}{\mathcal{M}}

\DeclareMathOperator{\polyfun}{n}

\newcommand{\charfun}{\ensuremath{\mathbbm 1}}

\DeclareMathOperator{\card}{card}

\allowdisplaybreaks

\begin{document}
\title{Unconditionality of periodic orthonormal   spline systems in $L^p$}
\author[K. Keryan]{Karen Keryan}
\address{Yerevan State University, Alex Manoogian 1, 0025, Yerevan, Armenia,
	American University of Armenia, Marshal Baghramyan 40, 0019, Yerevan,
Armenia}
\email{karenkeryan@ysu.am; kkeryan@aua.am}
\author[M. Passenbrunner]{Markus Passenbrunner}
\address{Institute of Analysis, Johannes Kepler University Linz, Austria, 4040 Linz, Altenberger Strasse 69}
\email{markus.passenbrunner@jku.at}
\date{August 29, 2017}
\begin{abstract}
Given any natural number $k$ and any dense point sequence $(t_n)$ on the torus
$\mathbb T$, we prove that the corresponding periodic orthonormal spline system of order $k$ is 
an unconditional basis in $L^p$ for $1<p<\infty$.
\end{abstract}
\maketitle 
\section{Introduction}
In this work, we are concerned with periodic orthonormal spline systems of
arbitrary order $k$ with arbitrary partitions. 
We let $(s_n)_{n=1}^\infty$ be a dense sequence of points in the
 torus $\mathbb T$
such that each point occurs at most $k$ times. Such point sequences are called
\emph{admissible}. 

For $n\geq k$, we define $\hat{\mathcal S}_n$ to be the space of polynomial splines
of order $k$ with grid points $(s_j)_{j=1}^n$. For each $n\geq k+1$, 
the space $\hat{\mathcal S}_{n-1}$ has codimension $1$ in $\hat{\mathcal S}_{n}$ and, therefore,
there exists a function $\hat f_{n}\in \hat {\mathcal S}_{n}$ that is
orthonormal to the space $\hat {\mathcal S}_{n-1}$. 
Observe that this function $\hat f_{n}$ is unique up to sign. 
In addition, let  $(\hat f_{n})_{n=1}^k$ be an orthonormal basis for $\hat{\mathcal S}_{k}$.
The
system of functions $(\hat f_{n})_{n=1}^\infty$ is called 
\emph{periodic} orthonormal spline system of order $k$ corresponding to the sequence
	$(s_n)_{n=1}^\infty$. 
{We remark that if a point $x$ occurs $m$ times in the
	sequence $(s_n)_{n=1}^\infty$ before index $N$, the space $\hat{\mathcal
	S}_N$
	consists of splines that are in particular 
$(k-1-m)$ times continuously differentiable 	at
	$x$, where here for $k-1-m\leq -1$ we mean that no restrictions at the point $x$ are
	imposed. This means that if $m=k$ and also $s_N=x$, we have
	$\mathcal{\hat S}_{N-1} = \mathcal{\hat S}_{N}$ and therefore it makes
no sense to consider non-admissible point sequences.}

The {main result of this article is the following}
\begin{thm}\label{thm:uncond}
Let $k\in\mathbb{N}$ and $(s_n)_{n\geq 1}$ be an admissible sequence of knots in
$\mathbb T$. Then the corresponding periodic orthonormal spline system of order $k$ is an unconditional basis in 
$L^p(\mathbb T)$ for every $1<p<\infty$.
\end{thm}

{This is the periodic version of the main result in \cite{Passenbrunner2014}. We
now give a few comments on the history of this result. We can similarly define
the spaces {$\mathcal S_n$} corresponding to an admissible
point sequence $(t_n)$ on the interval $[0,1]$.}
A celebrated result of A. Shadrin \cite{Shadrin2001} states that the orthogonal
projection operator onto those 
spaces $\mathcal S_n$ is bounded on $L^\infty[0,1]$ by a constant that depends only on the spline order $k$.
As a consequence, $(f_n)_{n}$ {(also similarly defined to $\hat f_n$)} is a
{Schauder} basis in $L^p[0,1],$ $1\leq p<\infty$ and {in the space of continuous
functions $C[0,1]$}.  
There are various results on the unconditionality of spline systems restricting
either the spline order $k$ or the partition $(t_n)_{n\geq 0}$. The first result
in this direction is \cite{Bockarev1975}, who proves  that the classical
Franklin system---that is orthonormal spline systems of order $2$ corresponding
to dyadic knot sequence {$(1/2,1/4,3/4,1/8,3/8,\ldots)$}---is an
unconditional basis in $L^p[0,1],\ 1<p<\infty$. This argument was extended in
\cite{Ciesielski1975} to prove unconditionality of orthonormal spline systems of
arbitrary order, but still restricted to dyadic knots. Considerable effort has
been made in the past to weaken the restriction to dyadic knot sequences. In the
series of papers \cite{GevorkyanKamont1998, GevorkyanSahakian2000, GevKam2004}
this restriction was removed step-by-step for general Franklin systems, with the
final result that it was shown for each admissible point sequence $(t_n)_{n\geq
0}$ with parameter $k=2$, the associated general Franklin system forms an
unconditional basis in $L^p[0,1]$, $1<p<\infty$. Combining the methods used in
\cite{GevorkyanSahakian2000, GevKam2004} with some new inequalities from
\cite{PassenbrunnerShadrin2013} it was  proved in \cite{Passenbrunner2014} that
non-periodic orthonormal spline systems are unconditional bases in $L^p[0,1],$
$1<p<\infty$, for any spline order $k$ and any admissible point sequence
{$(t_n)$}.

The periodic analogue of Shadrin's theorem can be obtained from Shadrin's result
\cite{Shadrin2001} using \cite{deBoor2012}.  
Alternatively,  \cite{Passenbrunner2017} gives a direct proof. In case of dyadic knots, J. Domsta \cite{Domsta1976}
obtained exponential decay for the inverse of the Gram matrix of periodic
B-splines, which {were exploited} to prove the unconditionality of the periodic
orthonormal spline systems     with dyadic knots in $L^p$ for $1<p<\infty.$ In
\cite{Keryan2005} it was proved that for any admissible point sequence the
corresponding periodic Franklin system {(i.e. the case $k=2$)} 
forms an unconditional basis in $L^p[0,1],\ 1<p<\infty.$
Here we obtain an estimate for {general} periodic orthonormal spline functions, which combined with the methods developed  in \cite{GevKam2004} yield the unconditionality of  periodic orthonormal spline systems in $L^p(\mathbb{T})$.

{The main idea of the proofs of $(f_n)$ or $(\hat f_n)$ being an
	unconditional basis in $L^p$, $p\in (1,\infty)$ in the articles 
	\cite{GevKam2004, Keryan2005, Passenbrunner2014} is that corresponding
	to one single function $f_n$, a grid point interval is associated on
	which most of the mass of $f_n$ is concentrated. In case of Haar
	functions $h_n$, its support splits into two intervals $I$ and $J$,
	where on $I$, the function $h_n$ is positive and on $J$, $h_n$ is
	negative. As the associated interval, we could just use the largest one
of $I$ and $J$.}

The organization of the present article is as follows. In Section \ref{sec:prel}, we give some preliminary results 
concerning polynomials, splines and non-periodic orthonormal spline functions. Section \ref{sec:periodic} develops crucial estimates for the periodic orthonormal spline functions $\hat f_n$ and gives several relations between $\hat f_n$  and its non-periodic counterpart. 
 In Section \ref{sec:techn} we prove a few technical lemmata used in the proof of Theorem \ref{thm:uncond} and Section \ref{sec:main} finally proves Theorem \ref{thm:uncond}.

We remark that the results and most of the proofs in Sections \ref{sec:techn}
and \ref{sec:main} follow closely \cite{GevKam2004}. Let us also remark that the
proof of {the crucial} Lemma \ref{lem:techn2} is new and  {much} shorter than it was in  \cite{GevKam2004}.
\section{Preliminaries}\label{sec:prel}
Let $k$ be a positive integer. The parameter $k$ will always be used for the order of the underlying polynomials or splines. We use the notation $A(t)\sim B(t)$ to indicate the existence of two constants $c_1,c_2>0$ that depend only on $k$, such that $c_1 B(t)\leq A(t)\leq c_2 B(t)$ for all $t$, where $t$ denotes all implicit and explicit dependences that the expressions $A$ and $B$ might have. If the constants $c_1,c_2$ depend on an additional parameter $p$, we write this as $A(t)\sim_p B(t)$. Correspondingly, we use the symbols $\lesssim,\gtrsim,\lesssim_p,\gtrsim_p$.  For a subset $E$ of the real line, we denote by $|E|$ the Lebesgue measure of $E$ and by $\charfun_E$ the characteristic function of $E$.

We will need the classical Remez inequality:
\begin{thm}[Remez] \label{thm:remez}
	Let $V\subset \mathbb R$ be a compact interval in $\mathbb R$ and $E\subset V$ a measurable
	subset. Then, for all polynomials $p$ of order $k$ on $V$,
	\begin{equation*}
		\| p \|_{L_\infty(V)} \leq \bigg( 4 \frac{|V|}{|E|}\bigg)^{k-1} \| p
		\|_{L_\infty(E)}.
	\end{equation*}
\end{thm}
This immediately yields the following corollary{:}
\begin{cor} \label{cor:remez}
Let $p$ be a polynomial of order  $k$ on a compact interval $V\subset \mathbb R$. Then
\begin{equation*}
	\big|\big\{ x \in V : |p(x)| \geq 8^{-k+1} \|p\|_{L_\infty(V)} \big\}\big| \geq |V|/2.
\end{equation*}
\end{cor}
\begin{proof}
 	This is a direct application of Theorem \ref{thm:remez} with the set 
$E :=  \{x\in V : |p(x)| \leq 8^{-k+1}\|p\|_{L^\infty(V)} \}$.  
\end{proof}

Let
\begin{equation}\label{eq:part}
	\mathcal
	T=(0=\tau_{-k}=\dots=\tau_{-1}<\tau_{0}\leq\dots\leq\tau_{n-1}<\tau_{n}=\dots=\tau_{n+k-1}=1)
\end{equation}
be a partition of $[0,1]$ consisting of knots of multiplicity at most $k$, that
means $\tau_i<\tau_{i+k}$ for all $0\leq i\leq n-1$. Let {$\mathcal{S}_{\mathcal
T}$} be the space of polynomial splines of order $k$ with knots $\mathcal
T$.  The basis of $L^\infty$-normalized B-spline functions in
{$\mathcal{S}_{\mathcal{T}}$} is denoted by $(N_{i,k})_{i=-k}^{n-1}$ 
or for short $(N_{i})_{i=-k}^{n-1}$. Corresponding to this basis, there exists a
biorthogonal basis of {$\mathcal{S}_{\mathcal{T}}$}, which is denoted by
$(N_{i,k}^*)_{i=-k}^{n-1}$ or $(N_{i}^*)_{i=-k}^{n-1}$.
Moreover, we write $\nu_i = \tau_{i+k}-\tau_i=|\supp N_i|$. We continue with recalling a few important results for B-splines $N_i$ and their dual functions $N_i^*$.

\begin{thm}[Shadrin \cite{Shadrin2001}] \label{thm:shadrin}
	Let $P$ be the orthogonal projection operator onto {$\mathcal S_{\mathcal
	T}$} with respect to the canonical inner product in $L^2[0,1]$.
	Then, there exists a constant $C_k$ depending only on the spline order
	$k$ such that 
	\begin{equation*}
		\| P : L^\infty[0,1] \to L^\infty[0,1] \| \leq C_k.
	\end{equation*}
\end{thm}

\begin{prop}[B-spline stability]\label{prop:lpstab} 
	Let $1\leq p\leq \infty$ and $g=\sum_{j=-k}^{n-1} a_j N_j$ be a linear
	combination of B-splines. Then,
\begin{equation}\label{eq:lpstab}
|a_j|\lesssim |L_j|^{-1/p}\|g\|_{L^p(L_j)},\qquad -k\leq j\leq n-1,
\end{equation}
where $L_j$ is a subinterval $[\tau_i,\tau_{i+1}]$ of $[\tau_j,\tau_{j+k}]$ of maximal length. Additionally,
\begin{equation}\label{eq:deboorlpstab}
	\|g\|_p\sim \Big(\sum_{j=-k}^{n-1} |a_j|^p \nu_j\Big)^{1/p}=\|
	(a_j\nu_j^{1/p})_{j=-k}^{n-1}\|_{\ell^p}.
\end{equation}
Moreover, if $h=\sum_{j=-k}^{n-1} b_j N_j^*$,
\begin{equation}
	\|h\|_p\sim\Big(\sum_{j=-k}^{n-1} |b_j|^p
	\nu_j^{1-p}\Big)^{1/p}=\|(b_j\nu_j^{1/p-1})_{j=-k}^{n-1}\|_{\ell^p}.
\label{eq:lpstabdual}
\end{equation}
\end{prop}

The two inequalites \eqref{eq:lpstab} and \eqref{eq:deboorlpstab} are Lemma 4.1
and Lemma 4.2 in \cite[Chapter 5]{DeVoreLorentz1993}, respectively. Inequality
\eqref{eq:lpstabdual} is a consequence of Theorem \ref{thm:shadrin}.
For a deduction of  the lower estimate in \eqref{eq:lpstabdual} from this result, 
see \cite[Property P.7]{Ciesielski2000}. The proof of the upper estimate uses a
simple duality argument which we shall present here:

\begin{proof} [Proof of the upper estimate in \eqref{eq:lpstabdual}]
	Just consider the case $p<\infty$ and w.l.o.g. we assume $b_j\geq 0$. 
	Let $N_{j,p} = \nu_j^{-1/p} N_j$ be the $p$-normalized B-spline
	function and correspondingly $N_{j,p}^* = \nu_j^{1/p} N_j^*$ be the
	$p$-normalized dual B-spline function. By definition, the system
	$N_{j,p}^*$ forms a dual basis to the system of functions $N_{j,p}$. By
	choosing $p' = p/(p-1)$ and $\alpha = 2/p'$ (so $2-\alpha = 2/p$), we
	obtain by \eqref{eq:deboorlpstab}
	\begin{align*}
		\sum_j b_j^2 &= \langle \sum_j b_j^\alpha N_{j,p}^*, \sum_j
		b_j^{2-\alpha} N_{j,p}\rangle \leq \| \sum_j b_j^{2-\alpha} N_{j,p} \|_p \|
		\sum_j b_j^\alpha N_{j,p}^* \|_{p'} \\
		&= \|\sum_j b_j^{2-\alpha} \nu_j^{-1/p} N_j\|_p \|\sum_j b_j^\alpha \nu_j^{1/p}
		N_j^*\|_{p'} \\
		&\lesssim \Big(\sum_j b_j^2 \Big)^{1/p} \|\sum_j
		b_j^\alpha\nu_j^{1/p}N_j^*\|_{p'}.
	\end{align*}
	So we get
	\begin{equation}
		\label{eq:duallpstab_aux}
		\Big(\sum_j b_j^2 \Big)^{1/p'} \lesssim \|\sum_j b_j^\alpha
		\nu_j^{1/p}N_j^*\|_{p'}.
	\end{equation}
	Setting $a_j = b_{j}^\alpha \nu_j^{1/p}$, we see that  
	$b_j^2=(a_j\nu_j^{-1/p})^{2/\alpha} = a_j^{p'} \nu_j^{-p'/p} =
	a_j^{p'}\nu_j^{1-p'}$ and therefore, we may write
	\eqref{eq:duallpstab_aux} as 
	\begin{equation*}
		\Big(\sum_j a_j^{p'} \nu_j^{1-p'}\Big)^{1/p'} \lesssim	\|
		\sum_j
		a_j N_j^* \|_{p'},
	\end{equation*}
	which is the upper estimate in \eqref{eq:lpstabdual}.
\end{proof}

It can be shown that Shadrin's theorem actually implies the following
estimate on the B-spline Gram matrix inverse:
\begin{thm}[\cite{PassenbrunnerShadrin2013}]\label{thm:maintool}
Let $k\in\mathbb{N}$, the partition $\mathcal T$ be defined as in
\eqref{eq:part} and $(a_{ij})_{i,j=-k}^{n-1}$ be the inverse of the Gram matrix
$(\langle N_{i},N_{j}\rangle)_{i,j=-k}^{n-1}$ of B-spline functions $N_i=N_{i,k}$ of order $k$ 
corresponding to the partition $\mathcal T$. Then, 
\[
|a_{ij}|\leq C \frac{q^{|i-j|}}{|\operatorname{conv}(\supp N_i \cup \supp N_j)|},\qquad
-k\leq i,j\leq n-1,
\]
where the constants $C>0$ and $0<q<1$ depend only on the spline order $k$ and where 
by $\operatorname{conv}(U)$  for $U\subset [0,1]$ we  denote the smallest
 subinterval of $[0,1]$ that contains $U$. 
\end{thm}

Let $f\in L^p[0,1]$ for some $1\leq p<\infty$. Since the orthonormal spline system $(f_n)_{n\geq -k+2}$ is a basis in $L^p[0,1]$, we can write $f=\sum_{n=-k+2}^\infty a_n f_n$.
Based on this expansion, 
we define the \emph{maximal function} $Mf:=\sup_m \big| \sum_{n\leq m} a_n f_n \big|$.
Given a measurable function $g$, we denote by $\mathcal Mg$ the \emph{Hardy-Littlewood maximal function} of $g$ defined as
\[
\mathcal Mg(x):=\sup_{I\ni x} |I|^{-1} \int_I |g(t)|\dif t,
\]
where the supremum is taken over all intervals $I$ containing the point $x$.

A corollary of Theorem \ref{thm:maintool} is the following relation between $M$ and $\mathcal M$:
\begin{thm}[\cite{PassenbrunnerShadrin2013}]\label{thm:maxbound}
If $f\in L^1[0,1]$, we have
\[
Mf(t)\lesssim \mathcal M f(t),\qquad t\in[0,1].
\]
\end{thm}
\subsection{Orthonormal spline functions, {non-periodic case}}
This section recalls some facts about  orthonormal
spline functions $f_n = f_n^{(k)}$ for fixed $k\in\mathbb N$ and 
$n\geq 2$ induced by the admissible sequence $(t_n)$.

We consider again the mesh $\mathcal T$ from before:
\begin{align*}
	\mathcal T=(0=\tau_{-k}=\dots=\tau_{-1}<\tau_{0}&\leq\dots\leq\tau_{i_0} \\
&\leq\dots\leq\tau_{n-1}<\tau_{n}=\dots=\tau_{n+k-1}=1),
\end{align*}
where we singled out the point $\tau_{i_0}$ and 
the partition $\widetilde{\mathcal T}$ is defined to be the same as $\mathcal
T$, but with $\tau_{i_0}$ removed. In the same way we denote by $(N_i:-k\leq
i\leq n-1)$ 
the B-spline functions corresponding to $\mathcal T$ and by
$(\widetilde{N}_i:-k\leq i\leq n-2)$ the B-spline functions corresponding to $\widetilde{\mathcal T}$. B\"{o}hm's formula \cite{Boehm1980} gives us the following relationship between $N_i$ and $\widetilde{N}_i$:
\begin{equation}\label{eq:boehm}
\left\{
\begin{aligned}
\widetilde{N}_i(t)&=N_i(t)  &\text{if }-k\leq i\leq i_0-k-1, \\
\widetilde{N}_i(t)&=\frac{\tau_{i_0}-\tau_i}{\tau_{i+k}-\tau_i}N_i(t)+\frac{\tau_{i+k+1}-\tau_{i_0}}{\tau_{i+k+1}-\tau_{i+1}}N_{i+1}(t) &\text{if }i_0-k\leq i\leq i_0-1, \\
\widetilde{N}_i(t)&= N_{i+1}(t) &\text{if }i_0\leq i\leq n-2.
\end{aligned}
\right.
\end{equation}

In order to calculate the orthonormal spline function corresponding to the
partitions $\widetilde{\mathcal T}$ and $\mathcal T$, we first determine a
function $g\in\lin\{N_i:-k\leq i\leq n-1\}$ such that $g\perp \widetilde{N}_j$ for
all $-k\leq j\leq n-2$. The function  $g$ is of the form (up to a multiplicative
constant)
\begin{equation}
	\label{eq:defg}
	g=\sum_{j=i_0-k}^{i_0} \alpha_j N_j^*,
\end{equation}
where $(N_j^*:-k\leq j\leq n-1)$ is the biorthogonal system to the functions
$(N_i:-k\leq i\leq n-1)$
and
\begin{equation}\label{eq:alpha2}
\alpha_j=(-1)^{j-i_0+k}\Big(\prod_{\ell=i_0-k+1}^{j-1}\frac{\tau_{i_0}-\tau_{\ell}}{\tau_{\ell+k}-\tau_{\ell}}\Big)\Big(\prod_{\ell=j+1}^{i_0-1}\frac{\tau_{\ell+k}-\tau_{i_0}}{\tau_{\ell+k}-\tau_{\ell}}\Big),\quad i_0-k\leq j\leq i_0.
\end{equation}
Alternatively, the coefficients $\alpha_j$ can be described by the recursion
\begin{equation}\label{eq:recalpha}
\begin{aligned}
\alpha_{i+1}\frac{\tau_{i+k+1}-\tau_{i_0}}{\tau_{i+k+1}-\tau_{i+1}}+\alpha_{i}\frac{\tau_{i_0}-\tau_i}{\tau_{i+k}-\tau_i}=0.
 \end{aligned}
\end{equation}

In order to give estimates for $g$ and a fortiori, for the normalized function
$f=g/\|g\|_2$, we assign to each function $g$ a characteristic interval
that is a grid point interval $[\tau_i,\tau_{i+1}]$ and lies in the proximity of
the newly inserted point $\tau_{i_0}$: 
\begin{defin}[\cite{Passenbrunner2014}, Characteristic interval for non-periodic
	sequences]\label{def:characteristic}
Let $\mathcal T,\widetilde{\mathcal T}$ be as above and $\tau_{i_0}$ be the new point in $\mathcal T$ that is not present in $\widetilde{\mathcal T}$. We define the \emph{characteristic interval $J$ corresponding to $\tau_{i_0}$} as follows. 
\begin{enumerate}
\item 
Let 
\[
\Lambda^{(0)}:=\{i_0-k\leq j\leq i_0 : |[\tau_j,\tau_{j+k}]|\leq 2\min_{i_0-k\leq \ell\leq i_0}|[\tau_\ell,\tau_{\ell+k}]| \}
\]
be the set of all indices $j$ for which the corresponding support of the B-spline function $N_j$ is approximately minimal. Observe that $\Lambda^{(0)}$ is nonempty.
\item Define
\[
\Lambda^{(1)}:=\{j\in \Lambda^{(0)}: |\alpha_j|=\max_{\ell\in \Lambda^{(0)}} |\alpha_\ell|\}.
\]
For an arbitrary, but fixed index $j^{(0)}\in \Lambda^{(1)}$, set $J^{(0)}:=[\tau_{j^{(0)}},\tau_{j^{(0)}+k}]$.
\item The interval $J^{(0)}$ can now be written as the union of $k$ grid intervals
\[
J^{(0)}=\bigcup_{\ell=0}^{k-1}[\tau_{j^{(0)}+\ell},\tau_{j^{(0)}+\ell+1}]\qquad\text{with }j^{(0)}\text{ as above}.
\]
We define the \emph{characteristic interval} $J=J(\tau_{i_0})$ to be one of the above $k$ intervals that has maximal length.
\end{enumerate}
\end{defin}

Using this definition of $J$, we recall the following estimates for $g$:

\begin{lem}[\cite{Passenbrunner2014}]\label{lem:orthsplineJinterval}Let $\mathcal T,\,\widetilde{\mathcal
	T}$ be as above and $g=\sum_{j=i_0-k}^{i_0}\alpha_j N_j^*=\sum_{j=-k}^{n-1}
	w_jN_j$ be the function from \eqref{eq:defg}, where the coefficients
	$(w_j)$ are defined through this equation.
Moreover, let $f=g/\|g\|_2$  be the $L^2$-normalized orthogonal spline function corresponding to the mesh point $\tau_{i_0}$.
Then,
\[
\|g\|_{L^p(J)}\sim\|g\|_p\sim |J|^{1/p-1},\qquad 1\leq p\leq \infty,
\]
and therefore
\[
\|f\|_{L^p(J)}\sim\|f\|_p\sim |J|^{1/p-1/2},\qquad 1\leq p\leq \infty,
\]
where $J$ is the characteristic interval associated to the point $\tau_{i_0}$ 
given in Definition \ref{def:characteristic}.

Additionally, if $d_{\mathcal T}(z)$ denotes the number of grid points from
$\mathcal T$ that lie between $J$ and $z$ including $z$ and endpoints of $J$,
then there exists a $q\in (0,1)$ depending only on $k$ such that
\begin{equation}\label{eq:wj}
	|w_j|\lesssim \frac{q^{d_{\mathcal
	T}(\tau_j)}}{|J|+\dist(\supp N_j,J)+\nu_j}\quad\text{for all }-k \leq
	j\leq n-1.
\end{equation}
Moreover, if $x<\inf J$, we have
\begin{equation}\label{eq:phiplinks}
\|f\|_{L^p(0,x)}
\lesssim \frac{q^{d_{\mathcal T}(x)}|J|^{1/2}}{(|J|+\dist(x,J))^{1-1/p}},\qquad 1\leq p\leq \infty.
\end{equation}
Similarly, for $x>\sup J$,
\begin{equation}\label{eq:phiprechts}
\|f\|_{L^p(x,1)}
\lesssim \frac{q^{d_{\mathcal T}(x)}|J|^{1/2}}{(|J|+\dist(x,J))^{1-1/p}},\qquad 1\leq p\leq \infty.
\end{equation}
\end{lem}

\subsection{Combinatorics of characteristic intervals}\label{sec:comb}
We additionally have a combinatorial lemma concerning the collection of
characteristic intervals corresponding to all grid points of 
an admissible sequence  $(t_n)$ of points and  the corresponding orthonormal spline functions 
$(f_n)_{n=-k+2}^\infty$ of order $k$.
For $n\geq 2$, the  partitions $\mathcal T_n$ associated to $f_n$ are defined to
consist of the grid points $(t_j)_{j=-1}^{n}$, the knots $t_{-1}=0$ and $t_0=1$ having
both multiplicity $k$ in $\mathcal T_n$ and we enumerate them as 
\begin{align*}
\mathcal T_n=(0=\tau_{n,-k}&=\dots=\tau_{n,-1}<\tau_{n,0}\leq&\\
&\leq\dots\leq\tau_{n,n-1}<\tau_{n,n}=\dots=\tau_{n,n+k-1}=1).
\end{align*}

If $n\geq 2$, we denote by $\ J_n^{(0)}$ and $J_n$ the characteristic intervals
$J^{(0)}$ and $J$ from Definition \ref{def:characteristic} associated to the new
grid point $t_n$, which is defined to be the characteristic interval associated
to $(\mathcal T_{n-1},\mathcal T_n)$. If $n$ is in the range $-k+2\leq n\leq 1$,
we additionally set $J_n:=[0,1]$.

\begin{lem}[\cite{Passenbrunner2014}]\label{lem:jinterval}
	Let $V$ be an arbitrary subinterval of $[0,1]$ and let $\beta >0$. Then there exists
	a constant $F_{k,\beta}$ only depending on $k$ and $\beta$ such that
\[
\card\{n:J_n\subseteq V, |J_n|\geq\beta|V|\} \leq F_{k,\beta},
\]
where $\card E$ denotes the cardinality of the set $E$.
\end{lem}
\section{Periodic splines}\label{sec:periodic}
In this section, we give estimates for periodic orthonormal spline functions $(\hat f_n)$ 
{similar to the ones contained in Lemma \ref{lem:orthsplineJinterval} for non-periodic
orthonormal splines.}
The main difficulty in proving such estimates is that we do not have a periodic version
of Theorem \ref{thm:maintool} at our disposal. Instead, we estimate the
differences between $\hat f_n$ and two suitable non-periodic orthonormal spline
functions $f_n$.

Let {$n\geq k$} and $(\hat N_i)_{i=0}^{n-1}$ be periodic B-spline functions {of order $k$} with arbitrary admissible grid
$(\sigma_j)_{j=0}^{n-1}$ on $\mathbb T$ canonically identified with $[0,1)$:
\[
	\hat{\mathcal T} = (0 \leq \sigma_0 \leq \sigma_1 \leq \cdots \leq
	\sigma_{n-2} \leq \sigma_{n-1} < 1).
\]
Moreover, let  $(\hat N_i^*)_{i=0}^{n-1}$ be the dual basis to $(\hat
N_i)_{i=0}^{n-1}$ and $\hat {\mathcal S}_{\hat {\mathcal T}}$ be the linear span of
$(\hat N_i)_{i=0}^{n-1}$.
First, we recall that we have a periodic version of Shadrin's theorem:
\begin{thm}
	\label{thm:boundedperiodic}
	Let $\hat P$ be the orthogonal projection operator onto  $\hat {\mathcal S}_{\hat {\mathcal T}}$ with respect to the canonical inner product in $L^2(\mathbb
	T)$.
	Then, there exists a constant $C_k$ depending only on the spline order
	$k$ such that 
	\begin{equation*}
		\| \hat P : L^\infty(\mathbb T) \to L^\infty(\mathbb T) \| \leq C_k.
	\end{equation*}
\end{thm}
We refer to the articles
\cite{Shadrin2001,deBoor2012} for a proof of this result for infinite knot
sequences on the real line, which then can be carried over to $\mathbb T$.
Alternatively, we refer to \cite{Passenbrunner2017} for a direct proof.

Next,  note that B-spline stability carries over to the periodic
setting:
\begin{prop}\label{prop:periodicstability}
	Let $n\geq 2k$ and $1\leq p\leq \infty$.
	Then, for $g= \sum_{j=0}^{n-1} a_j \hat N_j$, we have 
	\begin{equation*}
		\| g \|_p \sim \Big( \sum_{j=0}^{n-1} |a_j|^p |\supp \hat
		N_j|\Big)^{1/p} = \| (a_j \cdot |\supp \hat
		N_j|^{1/p})_{j=0}^{n-1}\|_{\ell^p}.
	\end{equation*}
\end{prop}

If we define the matrix $(\hat a_{ij})_{i,j=0}^{n-1} = (\langle \hat N_i^*,
\hat N_j^*\rangle)_{i,j=0}^{n-1}$, we have the following geometric decay
inequality, which is a
consequence of Theorem \ref{thm:boundedperiodic} on
the uniform boundedness of the periodic orthogonal spline
projection operator:
\begin{prop}
	\label{prop:geom_decay_periodic}
	Let  $n\geq 2k$. Then,
	there exists a constant $q\in (0,1)$ depending only on the spline order $k$ such that 
	\begin{equation*}
		|\hat a_{ij}| \lesssim \frac{q^{\hat d(i,j)}}{\max(|\supp \hat N_i|,
		|\supp \hat N_j|)}, \qquad 0\leq i,j\leq n-1,
	\end{equation*}
	where $\hat d$ is the periodic distance function on $\{0,\ldots, n-1\}$.
\end{prop}
The proof of this proposition follows along the same lines as in the
non-periodic case, where B-spline stability and Demko's theorem \cite{Demko1977} on the geometric
decay of inverses of band matrices is used. Its proof in the non-periodic case
can be found in \cite{Ciesielski2000}.

Observe that the estimate contained in this proposition for periodic splines is
not as good as the one from Theorem \ref{thm:maintool} for non-periodic splines
due to the {different} term in the denominator. 
Next, we also get stability of the periodic dual B-spline functions $(N_i^*)$:
\begin{prop}
	Let $n\geq 2k$, $1\leq p\leq\infty$ and $h=\sum_{j=0}^{n-1} b_j \hat N_j^*$, then
	\begin{equation*}
		\| h \|_p \sim \Big( \sum_{j=0}^{n-1} |b_j|^p |\supp \hat
		N_j|^{1-p}\Big)^{1/p} = \| (b_j\cdot |\supp \hat N_j|^{1/p -
		1})_{j=0}^{n-1}\|_{\ell^p}.
	\end{equation*}
\end{prop}
\begin{proof}
 {We only prove the assertion for $p\in (1,\infty)$. The boundary cases follow
from obvious modifications of the proof.}
	By Propositions \ref{prop:geom_decay_periodic},
	\ref{prop:periodicstability}, and H\"older's inequality
	\begin{align*}
		\Big\| \sum_j a_j \nu_j^{1/p'}\hat N_j^* \Big\|_p^p &= \Big\|
		\sum_j a_j \nu_j^{1/p'} \sum_i \hat a_{ij} \hat N_i \Big\|_p^p \\
		&= \Big\| \sum_i \Big( \sum_j a_j\nu_j^{1/p'} \hat a_{ij}\Big)
		\hat N_i\Big\|_p^p \\
		&\leq \sum_i \Big|\sum_j a_j\nu_j^{1/p'} \hat a_{ij} \Big|^p
		\nu_i \\
		&\lesssim \sum_i \Big(\sum_j |a_j| \nu_j^{1/p'}\nu_i^{1/p} \frac{q^{\hat
		d(i,j)}}{\max(\nu_i,\nu_j)}\Big)^p \\
		&\leq \sum_i \Big( \sum_j |a_j| q^{\hat d(i,j)}\Big)^p \leq \sum_i \Big( \sum_j |a_j|^p q^{\hat d(i,j)\frac{p}2}\Big)\cdot  \Big( \sum_j q^{\hat d(i,j)\frac{p}{2(p-1)}}\Big)^{p-1} 
				 \\
		&\lesssim  \sum_i  \sum_j |a_j|^p q^{\hat d(i,j)\frac{p}2}  \lesssim
		\|a\|_p^p.
	\end{align*}
	Setting $b_j=a_j\nu_j^{1/p'}$ yields the first inequality of dual
	B-spline stability. The other inequality is proved similarly to the
	result for the non-periodic case in Proposition~\ref{prop:lpstab}.
\end{proof}

\subsection{Periodic orthonormal spline functions}
We now consider the same situation as for the non-periodic case: Let 
\begin{equation*}
	\hat{\mathcal T} = (0 \leq \sigma_0 \leq \sigma_1 \leq \cdots \leq
	\sigma_{i_0} \leq \cdots \leq
	\sigma_{n-2} \leq \sigma_{n-1} < 1)
\end{equation*}
be a partition of $\mathbb T$ canonically identified with $[0,1)$ 
and $\widetilde{\hat{\mathcal T}}$ be the same partition, but with $\sigma_{i_0}$
removed.
Similarly, we denote by $(\hat N_j)_{j=0}^{n-1}$ the periodic B-spline
function{s of order $k$} with respect to $\hat{\mathcal T}$ and by 
$(\tilde{\hat {N_j}})_{j=0}^{n-2}$
the periodic B-spline functions  {of order $k$} with respect to $\widetilde{\hat{\mathcal T}}$.
Here, we use the notation of periodic extension of the sequence
$(\sigma_j)_{j=0}^{n-1}$, i.e. $\sigma_{rn+j} = r+\sigma_j$ for $j\in\{0,\ldots,n-1\}$ and
$r\in\mathbb Z$ and in
the subindices of the B-spline functions, we take the indices modulo $n$.

In order to calculate the periodic orthonormal spline functions corresponding to
the above grids, we determine a function $\hat g\in \lin \{\hat N_i : 0\leq
i\leq n-1\}$ such that $\hat g\perp \tilde{\hat{N_j}}$ for all $0\leq j\leq n-2$.
That is, we assume that $\hat g$ is of the form 
\begin{equation*}
	\hat g = \sum_{j=0}^{n-1} \hat \alpha_j \hat N_j^*,
\end{equation*}
where $(\hat N_j^* : 0\leq j\leq n-1)$ is the system biorthogonal to the
functions $(\hat N_i: 0\leq i\leq n-1)$ and $\hat\alpha_j = \langle g,\hat
N_j\rangle$. In order for $\hat g$ to be orthogonal to
$\tilde{\hat{N_j}}$ for $0\leq j\leq n-2$, it has to satisfy the identities
\begin{equation*}
	0 = \langle\hat{g},\tilde{\hat{N_i}}\rangle = \sum_{j=0}^{n-1} \hat\alpha_j
	\langle \hat N_j^*, \tilde{\hat {N_i}}\rangle ,\qquad 0 \leq i\leq n-2.
\end{equation*}
We can look at the indices $j$ here periodically meaning that 
$\hat{\alpha}_j\neq 0$ only for $j\in\{i_0-k,\ldots,i_0\}$.
Observing that the formulas for B-splines are  local and thus, we are
able to use 
formula  \eqref{eq:boehm}
for $\tilde{\hat{N}}_i$, to get the recursion
formula
\begin{equation}
	\label{eq:rec-hat-alpha}
	\hat\alpha_{i+1} \frac{\sigma_{i+k+1} -
	\sigma_{i_0}}{\sigma_{i+k+1}-\sigma_{i+1}} + \hat
	\alpha_i \frac{\sigma_{i_0}-\sigma_i}{\sigma_{i+k}-\sigma_i}=0, \qquad i_0-k \leq i\leq i_0-1.
\end{equation}
With the starting value
\begin{equation*}
	\hat \alpha_{i_0-k} = \prod_{\ell=i_0-k+1}^{i_0-1}
	\frac{\sigma_{\ell+k}-\sigma_{i_0}}{\sigma_{\ell+k}-\sigma_\ell},
\end{equation*}
we get the explicit formula
\begin{equation}\label{eq:alphaperiodic}
	\hat\alpha_j = (-1)^{j-i_0+k} \Big( \prod_{\ell=i_0-k+1}^{j-1}
	\frac{\sigma_{i_0}-\sigma_\ell}{\sigma_{\ell+k} - \sigma_\ell}\Big) \cdot \Big(
	\prod_{\ell=j+1}^{i_0-1}
	\frac{\sigma_{\ell+k}-\sigma_{i_0}}{\sigma_{\ell+k}-\sigma_\ell}\Big),\qquad
	i_0-k\leq j\leq i_0.
\end{equation}

Now, similarly to Definition \ref{def:characteristic}, we are able to define
characteristic intervals for periodic grids as follows:
\begin{defin}[Characteristic interval for periodic
	sequences]\label{def:characteristic_periodic}
	Let $\hat{\mathcal T},\widetilde{\hat{\mathcal T}}$ be as above and
	$\sigma_{i_0}$ be the new point in $\hat{\mathcal T}$ that is not present in
	$\widetilde{\hat{\mathcal T}}$. Under the restriction $n\geq 2k$, we define the \emph{characteristic
		interval $\hat{J}$ corresponding to $\sigma_{i_0}$} as follows. 
\begin{enumerate}
\item 
Let 
\[
\Lambda^{(0)}:=\{i_0-k\leq j\leq i_0 : |[\sigma_j,\sigma_{j+k}]|\leq 2\min_{i_0-k\leq
\ell\leq i_0}|[\sigma_\ell,\sigma_{\ell+k}]| \}
\]
be the set of all indices $j$ in the vicinity of the index $i_0$ 
for which the corresponding support of the periodic B-spline function $\hat{N}_j$ is approximately minimal. Observe that $\Lambda^{(0)}$ is nonempty.
\item Define
\[
\Lambda^{(1)}:=\{j\in \Lambda^{(0)}: |\hat\alpha_j|=\max_{\ell\in \Lambda^{(0)}}
|\hat\alpha_\ell|\}.
\]
For an arbitrary, but fixed index $j^{(0)}\in \Lambda^{(1)}$, set $\hat
J^{(0)}:=[\sigma_{j^{(0)}},\sigma_{j^{(0)}+k}]$.
\item The interval $\hat J^{(0)}$ can now be written as the union of $k$ grid intervals
\[
\hat J^{(0)}=\bigcup_{\ell=0}^{k-1}[\sigma_{j^{(0)}+\ell},\sigma_{j^{(0)}+\ell+1}]\qquad\text{with }j^{(0)}\text{ as above}.
\]
Define the \emph{characteristic interval} $\hat J=\hat J(\sigma_{i_0})$ to be one of the above $k$ intervals that has maximal length.
\end{enumerate}
\end{defin}

\subsection{$L^p$ norms of $\hat g$}

\begin{prop}
	\label{prop:periodicpnorm}
	Let $n\geq 2k+2$. Then, 
\begin{equation*}
	\|\hat g\|_p \sim |\hat J|^{1/p-1},\qquad 1\leq p\leq \infty.
\end{equation*}
\end{prop}
\begin{proof}
We are able to arrange the periodic point sequence 
$(\sigma_j)_{j=0}^{n-1}$ such that $\sigma_0>0$ and $i_0=\lfloor n/2\rfloor$. Corresponding to this
point sequence, we define a non-periodic point sequence
$(\tau_{j})_{j=-k}^{n+k-1}$ by $\tau_j=\sigma_j$ for $j\in\{0,\ldots,n-1\}$,
$\tau_{-k} = \cdots =  \tau_{-1}=0$ and $\tau_n = \cdots = \tau_{n+k-1} = 1$.
With this choice and the
assumption $n\geq 2k+2$,
the conditions $i_0\geq k$ and $i_0\leq n-k-1$ are satisfied. Therefore, by
comparing \eqref{eq:alpha2} to \eqref{eq:alphaperiodic}, we get
$\alpha_j = \hat\alpha_j$ for $i_0-k\leq j\leq i_0$, which yields
\begin{equation*}
	\hat g = \sum_{j=i_0-k}^{i_0} \hat\alpha_j \hat N_j^*, \qquad g =
	\sum_{j=i_0-k}^{i_0} \hat\alpha_j N_j^*.
\end{equation*}
Also, comparing the two definitions of $J$ and $\hat J$, in the present case we
see that $|J|= |\hat J|$ and thus, we use B-spline stability to get
\begin{equation*}
	\| \hat g\|_p^p \sim \sum_{j=i_0-k}^{i_0} |\hat
	\alpha_j|^p |\supp N_j|^{1-p} \sim \|g\|_p^p \sim |\hat J|^{p-1}.
\end{equation*}
where the last equivalence follows from Lemma \ref{lem:orthsplineJinterval}.
\end{proof}

\begin{lem}\label{lem:fhatbi}
	Let $n\geq 2k+2$.
	If we write $\hat g= \sum_{i=0}^{n-1} \hat w_i \hat N_i$, we can estimate the
	coefficients $\hat w_i$ by
	\begin{equation*}
		|\hat w_i|\lesssim
		q^{\hat{d}(i,i_0)}\max_{i_0-k\leq j\leq i_0}
		\frac{1}{\max(|\supp \hat N_i|,|\supp \hat N_j|)}
	\end{equation*}
	where we take the index $j$ to be modulo $n$ and $\hat d$ is the
	periodic distance function on $\{0,\ldots,n-1\}$.
\end{lem}
\begin{proof}
	By looking at formula \eqref{eq:alphaperiodic}, we see that 
	$	|\hat\alpha_j | \leq 1$ for all $j$ and 
	therefore, by Proposition \ref{prop:geom_decay_periodic},
	\begin{equation*}
		|w_i|  =  \Big| \sum_{j=i_0-k}^{i_0} \hat\alpha_j \hat a_{ij}
		\Big|
		\lesssim \sum_{j=i_0-k}^{i_0} |\hat a_{ij}|\lesssim 
		\sum_{j=i_0-k}^{i_0} \frac{q^{\hat d(i,j)}}{\max(|\supp \hat N_i|,|\supp
	\hat N_j|)}.
	\end{equation*}
	This readily implies the assertion.
\end{proof}

	\begin{prop}\label{prop:lower_Lp_estimate_fhat}
		There exists an index $N(k)$ that depends only  on $k$ such that
		for all partitions $\hat {\mathcal T}$ with $n\geq N(k)$, we have
		\begin{equation*}
			\|\hat g\|_{L^p(\hat J)} \gtrsim |\hat J|^{1/p-1},\qquad
			p\in[1,\infty].
		\end{equation*}
	\end{prop}
	\begin{proof}
		Assuming again that $i_0 = \lfloor n/2\rfloor$ { and $n\geq
		2k+2$},
		we begin by considering the difference between the periodic
		function $\hat g$ to the non-periodic function $g$ corresponding
		to the partition 
$\mathcal T=(\tau_{j})_{j=-k}^{n+k-1}$ with $\tau_j=\sigma_j$ for $j\in\{0,\ldots,n-1\}$,
$\tau_{-k} = \cdots =  \tau_{-1}=0$ and $\tau_n = \cdots = \tau_{n+k-1} = 1$:
\begin{equation*}
	u := g-\hat g  = \sum_{j=-k}^{n-1} \beta_j
	N_j^*,
\end{equation*}
where the coefficients $\beta_j$ are chosen so that this equation is true. This
is possible since both $g$ and $\hat g$ are contained in the linear span of the functions
$(N_j^*)$.
By defining  the  set of boundary indices $B$ in $\mathcal T$ by 
\begin{equation*}
	B = \{-k,\ldots,-1\} \cup \{n-k,\ldots,n-1\}\subset \{-k,\ldots,n-1\},
\end{equation*}
we see that for $j\in B^c$,
\begin{equation*}
	\beta_j =\langle u,
N_j\rangle = \langle g-\hat g, N_j\rangle =\langle g, N_j\rangle - 
	\langle \hat g, \hat N_j\rangle =
	\alpha_j - \hat\alpha_j=0,
\end{equation*}
where the last equation follows from the fact that $\alpha_j=\hat\alpha_j$ for
all indices $j$ in our current definition of $\mathcal T$.
Therefore, the function $u = g-\hat g$ can be expressed
as
\begin{equation}
	u = \sum_{j\in B} \beta_j N_j^*.
\end{equation}
Now, we estimate the coefficients $\beta_j$ for {$j\in B$} by Lemma  \ref{lem:fhatbi}:
\begin{align*}
	|\beta_j| &= |\langle g-\hat g, N_j\rangle| =  
	|\langle \hat g, N_j\rangle|
	\\
	&= \Big|\sum_{i=0}^{n-1} \hat w_i \langle \hat N_i, N_j\rangle\Big| \lesssim 
	 \sum_{i=0}^{n-1} |\hat w_i| \cdot |\supp \hat N_i \cap \supp
	N_j| \\
	& \lesssim  \sum_{i=0}^{n-1}  q^{\hat{d}(i,i_0)}\max_{m=i_0-k}^{i_0}
		\frac{1}{\max(|\supp \hat N_i|,|\supp \hat N_m|)} \cdot|\supp \hat N_i \cap \supp
	N_j| \\
	&\leq  \sum_{i: |\supp \hat N_i \cap \supp N_j|>0 }
	q^{\hat d(i,i_0)} 
\end{align*}
and, since $j\in B=\{-k,\ldots,-1\}\cup \{n-k,\ldots, n-1\}$,
\begin{equation}\label{eq:beta}
	|\beta_j| \lesssim  q^{\hat d(0,i_0)}\lesssim  q^{n/2}
, \qquad j\in B.
\end{equation}
So, we estimate for $x\in \hat J$:
	\begin{align*}
		|u(x)| &= \Big|\sum_{j\in B} \beta_j N_j^*(x)\Big| = \Big|\sum_{j\in B}
		\beta_j \sum_{i=-k}^{n-1} a_{ij} N_i(x)\Big| \\
		& = \Big|\sum_{j\in B} \beta_j \sum_{i: \hat J\subset \supp N_i} a_{ij}
		N_i(x)\Big| \\
		&\lesssim\sum_{j\in B} |\beta_j| \max_{i:\hat J\subset \supp N_i} |a_{ij}|.
	\end{align*}
	So, by \eqref{eq:beta} and the estimate in Theorem
	\ref{thm:maintool} for the
	\emph{non-periodic} matrix $(a_{ij})$ 
	\begin{align*}
		|u(x)| \lesssim q^{n/2}
		\max_{i:\hat J\subset\supp N_i} \max_{j\in B}
		\frac{q^{|i-j|}}{h_{ij}},
	\end{align*}
	where $h_{ij} = |\operatorname{conv} ( \supp N_i \cup \supp N_j)|$.
	Since $\hat J\subset \supp N_i$ for the above indices $i$, we have
	$h_{ij}\geq |\hat J|=|J|$
	and therefore,
	\begin{equation*}
		|u(x)| = |(g-\hat g)(x)| \lesssim q^n |J|^{-1}.
	\end{equation*}
	This means that on $J$, we can estimate $\hat g$ from below: let $x\in J$ be a point
	such that $|g(x)| \geq  \| g \|_{L^\infty(J)}/2$, then $|g(x)| \gtrsim
	|J|^{-1}$ by Lemma \ref{lem:orthsplineJinterval} and we get
	\begin{align*}
		|\hat g(x)| &= |g(x) - (g(x) -  \hat g(x)) | \\ 
		&\geq |g(x)| - |g(x) - \hat g(x)|  \\
		&\geq C_1|J|^{-1} - C_2 |J|^{-1} q^n,
	\end{align*}
	where $C_1$ and $C_2$ are constants that only depend on $k$ and $q<1$.
	So there exists an index $N(k)$ such that for all $n\geq N(k)$  
	\begin{equation*}
		\| \hat g \|_{L^\infty(\hat J)} \gtrsim  |\hat J|^{-1}.
	\end{equation*}
	 Since $\hat g$ is a polynomial on
	$\hat J$, by Corollary \ref{cor:remez} we now get for any $p\in[1,\infty]$
	\[
		\| \hat g \|_{L^p(\hat J)} \gtrsim |\hat J|^{1/p - 1},
	\]
	 which is the assertion.
\end{proof}
\subsection{More estimates for $\hat g$}
We now change our point of view slightly and compare the function $\hat g$
with a non-periodic function $g$ where we shift the sequence
$\hat {\mathcal T}=(\sigma_j)_{j=0}^{n-1}$ in such a way that we split in the middle of a largest grid
point interval:
\[
	\sigma_0 = 1-\sigma_{n-1} = \frac{1}{2}\max_{0\leq j\leq n-1}
	\text{{$(\sigma_{j} -
	\sigma_{j-1})$}},
\]
and, like before, choose $\mathcal T=(\tau_j)_{j=-k}^{n+k-1}$ such that $\tau_j = \sigma_j$
for $j\in \{0,\ldots, n-1\}$ so that we have
\[
\tau_0 - \tau_{-1} =
\tau_n-\tau_{n-1}=\frac{1}{2}\max_{0\leq j\leq n-1} (\sigma_j -
\sigma_{j-1}).
\] 
We refer to this choice of $\mathcal T$ as the \emph{maximal splitting} of
$\hat{\mathcal T}$.
Similar to above, we define $\widetilde{\hat{\mathcal T}}$ and
$\widetilde{\mathcal T}$ to be the partitions $\hat{\mathcal T}$ and $\mathcal
T$ respectively, with the grid points $\sigma_{i_0}$ and $\tau_{i_0}$ removed.

If we work under this assumption, it is not necessarily the case that $|J| =
|\hat J|$ {as there is the possibility that $J$ lies near $\tau_0$ or
$\tau_n$} , but we have 
\begin{prop}\label{prop:Jcompare}
	Let $J$ be the characteristic interval corresponding to the point
	sequences $(\mathcal T,\widetilde{\mathcal T})$ and $\hat J$ be the
	periodic one corresponding to $(\hat {\mathcal T},
	\widetilde{\hat{\mathcal T}})$ with the above maximal splitting. 
	Then 
	\begin{equation*}
		|J| \sim |\hat J|.
	\end{equation*}
\end{prop}
\begin{proof}
	The definition of the characteristic intervals $J$ (Definition
	\ref{def:characteristic}) and $\hat J$ (Definition~\ref{def:characteristic_periodic}) yield
	that 
	\begin{equation}
		\label{eq:J-J-hat}
		|J| \sim \min_{i_0-k\leq j\leq i_0} |\supp N_j|,\qquad |\hat
		J| \sim \min_{i_0-k\leq j\leq i_0} |\supp \hat N_j|,
	\end{equation}
	where the periodic indices are interpreted in the sense of  the usual
	periodic continuation of the subindices.
	Then, the very definition of the point sequence
	$\mathcal T$ in terms of the point sequence $\hat{\mathcal T}$ implies
	\begin{equation*}
		|\supp N_j| \leq |\supp \hat N_j|,\qquad -k\leq j\leq n-1,
	\end{equation*}
	so, in combination with \eqref{eq:J-J-hat},  we get the first inequality
	$	|J|\lesssim|\hat J|$.
	In order to show the converse inequality, we show
	\begin{equation}
		\label{eq:N-periodic-non-periodic}
		\min_{i_0-k\leq j\leq i_0} |\supp \hat N_j| \lesssim
		\min_{i_0-k\leq j\leq i_0} |\supp N_j|.
	\end{equation}
	We assume that $j_0$ is an index such that 
		$|\supp N_{j_0}| = \min_{i_0-k\leq j\leq i_0} |\supp N_j|$.
	If $j_0\notin B=\{-k,\ldots,-1\} \cup \{n-k,\ldots,n-1\}$, we even have 
	$	|\supp N_{j_0}| = |\supp \hat N_{j_0}|$.
	If $j_0\in B$, we have due to the choice of the  maximal
	splitting
	\begin{equation*}
		|\supp N_{j_0}| \geq \frac{1}{2}\max_{0\leq j\leq n-1}
		(\sigma_{j+1}-\sigma_j)
		\geq \frac{1}{2k} |\supp \hat N_{j}| 
	\end{equation*}
	for all indices $j$. So, in particular,
	\eqref{eq:N-periodic-non-periodic} holds. Thus we have shown the converse inequality
	$	|\hat J|\lesssim |J|$
	as well and the proof is complete.
\end{proof}
We also have the following relation between the dual B-spline coefficients
of $g$ and $\hat g$:
\begin{prop}
	\label{prop:alphacompare}
	For the maximal splitting, there exists a constant $c\sim 1$ such that
	for all $j\notin B$,
	\begin{equation*}
		\alpha_j = c\cdot\hat\alpha_j.
	\end{equation*}
\end{prop}
\begin{proof}
Comparing the recursion formulas \eqref{eq:recalpha} for $\alpha_j$ and
\eqref{eq:rec-hat-alpha} for $\hat{\alpha}_j$, we
	see that for $j\in \{i_0-k,\ldots, i_0-1\}$,
	\begin{equation}
		\label{eq:alphacomparison1}
		\frac{\hat{\alpha}_{j+1}}{\hat{\alpha}_j} =
		\frac{\alpha_{j+1}}{\alpha_j},\qquad \{j,j+1\}\subset B^c
	\end{equation}
	since by definition $\tau_i=\sigma_i$ for $0\leq i\leq n-1$.
	So, now take an arbitrary $j\in B^c$. Looking at the formulas for
	$\alpha_j$ and $\hat{\alpha}_j$ we write
	\begin{align*}
	\frac{\hat\alpha_j}{\alpha_j}&=    \Big(\prod_{\ell=i_0-k+1}^{j-1}
	\frac{\sigma_{i_0}-\sigma_\ell}{\tau_{i_0}-\tau_{\ell}}
\Big)\Big(\prod_{\ell=i_0-k+1}^{j-1}	
	\frac{\tau_{\ell+k}-
	\tau_{\ell}}{\sigma_{\ell+k} - \sigma_\ell}\Big) \cdot  
	\Big(\prod_{\ell=j+1}^{i_0-1}
	\frac{\sigma_{\ell+k}-\sigma_{i_0}}{\tau_{\ell+k}-\tau_{i_0}}
	\Big)\Big(\prod_{\ell=j+1}^{i_0-1}
	\frac{\tau_{\ell+k}-\tau_{\ell}}{\sigma_{\ell+k}-\sigma_\ell}\Big) .
	\end{align*}
Note that for every $s,t\in\{i_0-k+1,\ldots,i_0+k-1\}$ such that $0<s-t\leq k$  either  
$\sigma_s-\sigma_t=\tau_s-\tau_t $ or  $\sigma_s-\sigma_t>\tau_s-\tau_t $,
and the latter can only happen when $[\tau_{-1},\tau_0]$ or {$[$}$\tau_{n-1},\tau_n]$ 
is a subset of $[\tau_{t},\tau_s]$, so
$$\sigma_{s} - \sigma_t \geq \tau_{s}-\tau_t \geq
		\frac{1}{2}\max_{0\leq j\leq n-1}
		(\sigma_{j+1}-\sigma_j) \geq \frac{1}{2k}(\sigma_{s} -
		\sigma_t).$$
Hence we obtain $\sigma_s-\sigma_t \sim\tau_s-\tau_t.$
Therefore $\alpha_j \sim \hat{\alpha}_j$.  
	This, in combination with \eqref{eq:alphacomparison1} proves the proposition.
\end{proof}

\begin{prop}\label{prop:convinterval}
		Let $x\in [\sigma_\ell,\sigma_{\ell+1}]$. Then, there exists
		 an interval 
		$C=C(x)\subset \mathbb T$ which is minimal under the
		inclusion relation
		with
		\begin{equation*}
			\hat J\cup [\sigma_\ell, \sigma_{\ell+1}] \subset  C
		\end{equation*}
		such that if $K(C)$ is the number of grid points of 
		$\hat{\mathcal T}$ contained in $C$,
		\begin{equation*}
			|\hat g(x)| \lesssim \frac{\hat q^{K(C)}}{|C|},
		\end{equation*}
		where $\hat q\in (0,1)$ depends only on $k$.
	\end{prop}

	\begin{proof}
In order to estimate $\hat g$, we consider the difference   {$u:=g-c\cdot\hat g$} and $g$
separately, where $g$ is the orthogonal spline function corresponding to
$(\widetilde{\mathcal T},\mathcal
T)$ that arises from the maximal splitting {and $c\sim 1$ denotes the constant from
	Proposition \ref{prop:alphacompare}. We can write
\begin{equation*}
		u  = \sum_{j=-k}^{n-1} \beta_j N_j^*
\end{equation*}
for some coefficients $\beta_j$}. This
is possible since $g$ as well as $\hat g$ is contained in the linear span of
$(N_j^*)_{j=-k}^{n-1}$.
We calculate for $j\notin B= \{-k,\ldots,-1\} \cup \{n-k,\ldots,n-1\}$,
\begin{equation*}
	\beta_j =\langle g - c\cdot \hat g,
	N_j\rangle = \langle g, N_j\rangle - c\cdot
	\langle \hat g, \hat N_j\rangle =
	\alpha_j - c\cdot \hat{\alpha}_j=0,
\end{equation*}
where the last equality follows from Proposition \ref{prop:alphacompare}.
Therefore, the function $u = g-c\cdot \hat g$ can be expressed
as
\begin{equation}
	\label{eq:difference_nonperiodic_periodic}
	u = \sum_{j\in B} \beta_j N_j^*
\end{equation}
and its coefficients $\beta_j$ can be estimated by
\begin{align*}
	|\beta_j| &= |\langle g - c \hat g, N_j\rangle| \lesssim  |\langle g,
	N_j\rangle | + |\langle \hat g, N_j\rangle|
	\\
	&= \Big|\sum_{i=-k}^{n-1} w_i \langle N_i,N_j\rangle \Big| +
	\Big|\sum_{i=0}^{n-1} \hat  w_i \langle \hat N_i, N_j\rangle\Big| =:
	\Sigma_1 +
	\Sigma_2.
\end{align*}
Now, we estimate the term $\Sigma_1$ by using inequality \eqref{eq:wj} and the fact that
$j\in B$:
\begin{align*}
	\Sigma_1 &\leq \sum_{i=-k}^{n-1} |w_i| |\supp N_i \cap \supp N_j| \\
	&\lesssim \sum_{i=-k}^{n-1} \frac{q^{d_{\mathcal
	T}(\tau_i)}}{|J|+\dist(\supp N_i,J)+|\supp N_i|} |\supp N_i\cap\supp N_j| \\
	& \leq  \sum_{i: |\supp N_i\cap \supp N_j|>0} q^{d(i,i_0)}  \leq  q^{\hat d(0,i_0)}.
\end{align*}
The term $\Sigma_2$ is estimated similarly by using Lemma \ref{lem:fhatbi}:
\begin{align*}
	\Sigma_2 &\leq \sum_{i=0}^{n-1} |\hat w_i| \cdot |\supp \hat N_i \cap \supp
	N_j| \\
	& \lesssim  \sum_{i=0}^{n-1}  q^{\hat{d}(i,i_0)}\max_{i_0-k\leq m\leq
	i_0} \frac{1}{\max(|\supp \hat N_i|,|\supp \hat N_m|)} \cdot|\supp \hat N_i \cap \supp
	N_j| \\
	&\leq  \sum_{i: |\supp \hat N_i \cap \supp N_j|>0 }
	q^{\hat d(i,i_0)} \lesssim q^{\hat d(0,i_0)}.
\end{align*}
Combining the estimates for $\Sigma_1$ and $\Sigma_2$,
we get $|\beta_j|\lesssim  q^{\hat d(0,i_0)}$.
So, we continue to estimate $u(x)$ for $x\in [\tau_\ell, \tau_{\ell+1})$:
	\begin{align*}
		|u(x)| &= \Big|\sum_{j\in B} \beta_j N_j^*(x)\Big| = \Big|\sum_{j\in B}
		\beta_j \sum_{i=-k}^{n-1} a_{ij} N_i(x)\Big| \\
		& = \Big|\sum_{j\in B} \beta_j \sum_{i=\ell-k+1}^{\ell} a_{ij}
		N_i(x)\Big| \\
		&\leq \sum_{j\in B} |\beta_j| \max_{i=\ell-k+1}^\ell |a_{ij}|.
	\end{align*}
	So, by the above calculation and the estimate for the
	\emph{non-periodic} Gram matrix inverse  in Theorem
		\ref{thm:maintool},
	\begin{align*}
		|u(x)| \lesssim q^{\hat d(0,i_0)}
		\max_{i=\ell-k+1}^{\ell} \max_{j\in B}
		\frac{q^{|i-j|}}{h_{ij}},
	\end{align*}
	where $h_{ij} = |\operatorname{conv} ( \supp N_i \cup \supp N_j)|$.
	Since for $j\in B$, either $h_{ij}\geq \tau_0-\tau_{-1}$ or $h_{ij} \geq
	\tau_n-\tau_{n-1}$, we have by the defining property of the maximal
	splitting {that} $h_{ij}\geq \frac{1}{2}\max_{0\leq m\leq n-1}
	(\sigma_{m} - \sigma_{m-1})$, and therefore,
	\begin{equation}
		\label{eq:main}
		|u(x)| \lesssim  \frac{q^{\hat d(i_0,\ell)} }{\max_m
			(\sigma_m - \sigma_{m-1})},
	\end{equation}
	since we also have $\hat d(i_0,\ell) \leq \hat d(i_0,0) + \hat d(0,\ell)
	\leq \hat d(i_0,0) + 2k + \min_{j\in B, \ell-k+1\leq i\leq \ell} |i-j|.$ 
	Thus, we conclude  by Lemma \ref{lem:orthsplineJinterval} and
	\eqref{eq:main}
	\begin{equation*}
		|\hat g(x)| \leq c^{-1}|g(x)|  + |\hat g(x) - c^{-1}g(x)| 
		\lesssim 
		\frac{q^{d(i_0,\ell)}}{|\operatorname{conv}([\tau_\ell,\tau_{\ell+1}]\cup J)|} + 
		\frac{q^{\hat d(i_0,\ell)} }{\max_m ( \sigma_m -  \sigma_{m-1})},
	\end{equation*}
	which, with the use of Proposition \ref{prop:Jcompare} and the
	definitions of the characteristic intervals $J$ and $\hat J$,
	finishes the proof.  \end{proof}

	So, by defining the normalized orthonormal spline function $\hat f = \hat
	g/\|\hat g\|_2$, we immediately obtain 
	\begin{cor} \label{cor:geomdecay}
		Let $U$ be an arbitrary subset of $\mathbb T$. Then,
		\begin{equation*}
			\int_U |\hat f(x)|^p \dif x \lesssim |\hat J|^{p/2}
			\sum_{\ell : [ \sigma_\ell, \sigma_{\ell+1}]\cap U\neq \emptyset}
			\frac{\hat q^{pK(C( \sigma_\ell))}}{|C( \sigma_\ell)|^p} |U\cap
			[ \sigma_\ell, \sigma_{\ell+1}]|
		\end{equation*}
		where $\hat q\in(0,1)$ depends only on $k$.
	\end{cor}

	We will also need the pointwise estimate of the maximal spline projection
	operator by the Hardy-Littlewood maximal function in the periodic case, which is true in the
	non-periodic case by Theorem \ref{thm:maxbound}:
	\begin{thm}
		\label{thm:hardy-littlewood-periodic}
		Let $\hat P$ be the orthogonal projection onto  {$\hat{\mathcal
		S}_{\hat{\mathcal T}}$}. Then 
		\begin{equation*}
			|\hat P h(t)| \lesssim  \hat{\mathcal M} h(t),\qquad
			h\in L^1(\mathbb T),
		\end{equation*}
		where $\hat{\mathcal M} h(t) = \sup_{I\ni t}|I|^{-1}\int_I |h(y)|\dif
		y$ is the periodic Hardy-Littlewood maximal function operator
		and the $\sup$ is taken over all intervals $I\subset \mathbb T$
		containing the point $t$.
	\end{thm}
	\begin{proof}
		Let $h$ be such that $\supp h \subset
		[\sigma_\ell,\sigma_{\ell+1}]$ {for some
			$\ell\in\{0,\ldots,n-1\}$}. The
		first thing we show is that for any index $r$,
		\begin{equation*}
			\| \hat P h \|_{L^1[\sigma_r,\sigma_{r+1}]} \lesssim q^{\hat
			d(r,\ell)} \| h\|_{L^1}.
		\end{equation*}
		For this, we write in the case $t\in [\sigma_r,\sigma_{r+1}]$
		\begin{equation*}
			\hat P h(t) = \sum_{j: \supp \hat N_j\ni t} \sum_{i:\supp
				\hat N_i\supset [\sigma_\ell,\sigma_{\ell+1}]}
				\hat a_{ij} \langle h, \hat N_i\rangle
			\hat N_j(t).
		\end{equation*}
		After using Proposition \ref{prop:geom_decay_periodic} and a
		simple H\"{o}lder, this 
		is less than
		\begin{equation*}
			\|h\|_{L^1}\cdot \sum_{j: \supp \hat N_j\ni t} \sum_{i:\supp
				\hat N_i\supset [\sigma_\ell,\sigma_{\ell+1}]}
				\frac{q^{\hat d(i,j)}}{\max(|\supp \hat N_i|, |\supp
			\hat N_j|)}  \hat N_j(t).
		\end{equation*}
		Integrating this estimate over $[\sigma_r,\sigma_{r+1}]$, we get
		\begin{equation}
			\label{eq:localL1}
			\|\hat P h\|_{L^1[\sigma_r,\sigma_{r+1}]} \lesssim
			\|h\|_{L^1} q^{\hat d(\ell,r)}.
		\end{equation}
		The same can be proved for the non-periodic projection operator
		$P$, since here we can use the same estimates.

		Now, we take an arbitrary function $h$ and localize it by
		setting
		\begin{equation*}
			h_\ell = h\cdot \charfun_{[\sigma_\ell,\sigma_{\ell+1}]}.
		\end{equation*}
		We fix a point $t\in [\sigma_m,\sigma_{m+1}]$ and associate to $\hat
		P$ the non-periodic projection operator $P$ corresponding to
		the maximal splitting. Then
		\begin{equation}\label{eq:decompP}
			\hat P h(t) = P h(t) + \big(\hat Ph(t) - P h(t)\big).
		\end{equation}
		In order to show $\hat P h(t) \lesssim \hat {\mathcal M} h(t)$, we first
		recall that Theorem \ref{thm:maxbound} yields $|Ph(t)| \lesssim
		$ {$\mathcal M h(t)\leq$}
		$\hat{\mathcal M}h(t)$. For the
		second term $(\hat P - P)h$, we write
		\begin{equation*}
			(\hat P - P) h = \sum_{\ell=0}^{n-1} (\hat P -
			P) h_\ell
		\end{equation*}
		and prove an estimate for $g_\ell(t) := (\hat P - P)
		h_\ell$.
		Observe that 
		\begin{align*}
			g_\ell(t) = \sum_{i\in B} \langle g_\ell, N_i\rangle N_{i}^*(t)
			= \sum_{j=m-k+1}^m \sum_{i\in B} a_{ij}\langle g_\ell,
			N_i\rangle N_j(t),
		\end{align*}
		since the range of both $\hat P$ and $P$ is contained in the
		linear span of the functions $(N_i^*)$ {and $h_\ell - \hat P
		h_\ell$ and $h_\ell - Ph_\ell$ are both orthogonal to the 
	span of $N_i, i\notin B$}.
		Therefore, by using Theorem \ref{thm:maintool} for $a_{ij}$,
		\begin{equation*}
			|g_\ell(t)| \lesssim \sum_{j=m-k+1}^m \sum_{i\in B}
			\frac{q^{|i-j|}}{h_{ij}} \|g_\ell\|_{L^1(\supp N_i)}.
		\end{equation*}
		 {Consequently, by} \eqref{eq:localL1} and its non-periodic counterpart,
		\begin{equation}\label{eq:maxfun1}
			|g_\ell(t)| \lesssim \sum_{j=m-k+1}^m \sum_{i\in B}
			\frac{q^{|i-j|}}{h_{ij}} q^{\hat d(i,\ell)}
			\|h_\ell\|_{L^1}.
		\end{equation}
		Since we performed the maximal splitting for our periodic
		partition, we get that 
		\begin{equation*}
			h_{ij} \geq \frac{1}{2}\max_{\nu}(\sigma_\nu -
			\sigma_{\nu-1}),\qquad i\in B.
		\end{equation*}
		Denoting by $C_{\ell m}$ the convex set
		that contains $[\sigma_\ell,\sigma_{\ell+1}]\cup
		[\sigma_m,\sigma_{m+1}]$ containing  
		the least grid points, we get 
		\begin{equation*}
			h_{ij} \gtrsim \frac{|C_{\ell m}|}{\hat
			d(\ell,m)},\qquad i\in B.
		\end{equation*}
		Thus, we estimate \eqref{eq:maxfun1} by
		\begin{equation*}
			\sum_{j=m-k+1}^m \sum_{i\in B} q^{|i-j|+\hat
			d(i,\ell)}\hat d(\ell,m)
			\frac{\|h\|_{L^1[\sigma_\ell,\sigma_{\ell+1}]}}{|C_{m\ell}|}
			\lesssim \hat d(\ell,m) \max_{i\in B}(q^{|i-m| + \hat
			d(i,\ell)})\cdot\hat{\mathcal M}h(t)
		\end{equation*}
		{for all $t\in [\sigma_m,\sigma_{m+1}]$}.
		By the triangle inequality, $\hat d(\ell,m) \leq \hat d(i,m) +
		\hat d(i,\ell)\leq |i-m| + \hat d(i,\ell)$ and thus we can
		estimate further
		\begin{equation*}
			|g_\ell(t)|\lesssim \max_{i\in B} \alpha^{|i-m| + \hat d(i,\ell)}
			\hat{\mathcal M} h(t),
		\end{equation*}
		where $\alpha$ can be chosen as $q^{1/2}$.
		Summing this over $\ell$, we finally obtain
		\begin{equation*}
			|(\hat P - P) h(t)| \lesssim \alpha^{\hat d(0,m)}
			\hat{\mathcal{M}}h(t)	\leq \hat{\mathcal
			M}h(t),
		\end{equation*}
		which, in combination with \eqref{eq:decompP} and the result for
		$Ph(t)$ yields the assertion of the theorem.
	\end{proof}

\subsection{Combinatorics of characteristic intervals}\label{sec:comb_periodic}
Similarly to the non-periodic case we are able to analyze the combinatorics of
subsequent characteristic intervals.
Let   $(s_n)_{n=1}^\infty$ be an admissible sequence of points in $\mathbb T$
and  $(\hat f_n)_{n=1}^\infty$ be the corresponding periodic orthonormal spline functions of order $k$.
For $n\geq 1$, the  partitions $\hat{\mathcal T}_n$ associated to $\hat f_n$ are defined to
consist of the grid points $(s_j)_{j=1}^{n}$
and we enumerate them as 
\begin{align*}
	\mathcal {\hat T}_n=(0\leq \sigma_{n,0}\leq
\dots\leq\sigma_{n,n-1}<1).
\end{align*}

If $n\geq 2k$, we denote by $\hat  J_n^{(0)}$ and $\hat J_n$ the characteristic intervals
$\hat J^{(0)}$ and $\hat J$ from Definition \ref{def:characteristic_periodic} associated to the new
grid point $s_n$, which is defined to be the characteristic interval associated
to $(\hat{\mathcal T}_{n-1},\hat{\mathcal T}_n)$. 
For any $x\in \mathbb T$, let $C_n(x)$ be the {interval} from Proposition~\ref{prop:convinterval}
associated to $\hat J_n$.  We define $\hat d_n(x)$ to be
the number of grid points in $\hat{\mathcal T}_n$ between $x$ and $\hat J_n$
 contained in $C_n(x)$ counting $x$
and endpoints of $\hat J_n$. Moreover, for a subinterval $V$ of $\mathbb T$, we
denote 
$\hat d_n(V)=\min_{x\in V}\hat d_n(x)$. 

 An immediate consequence of the definition of $\hat J_n$ is that the
sequence of characteristic intervals $(\hat J_n)$ forms a nested collection of
sets, 
i.e., two sets in it are either disjoint or one is contained in the other.

Since the definition of the characteristic interval $\hat J_n$ only involves
local properties of the point sequence $(s_j)$ and the definition of $\hat J_n$
is the same as the definition of $J_n$ for any identification of $\mathbb T$
with $[0,1${$)$} that has the property that between the newly inserted point
$s_{i_0}$ and $0$ or $1$ are more than $k$ grid points of $\mathcal T_n$, we
also get the periodic version of Lemma \ref{lem:jinterval}.
\begin{lem}\label{lem:jinterval_periodic}
	Let $V$ be an arbitrary subinterval of $\mathbb T$ and let $\beta >0$. Then there exists
	a constant $F_{k,\beta}$ only depending on $k$ and $\beta$ such that
\[
\card\{n\geq 2k:\hat J_n\subseteq V, |\hat J_n|\geq\beta|V|\} \leq F_{k,\beta}.
\]
\end{lem}

Additionally, Lemma \ref{lem:jinterval_periodic} has the following corollary:
\begin{cor}\label{cor:J_n decay}
	Let $(\hat J_{n_i})_{i=1}^\infty$ be a decreasing sequence of characteristic
	intervals{, i.e. $\hat J_{n_{i+1}}\subseteq \hat J_{n_i}$}. Then, there exists a number $\kappa\in(0,1)$ and  a constant
	$C_k$, both depending only on $k$ such that 
	\begin{equation*}
		|\hat J_{n_i}| \leq C_k \kappa^{i}|\hat J_{n_1}|,\qquad i\in\mathbb N.
	\end{equation*}
\end{cor}
	\section{Technical estimates}\label{sec:techn}
The lemmas proved in this section are similar to the corresponding results in
\cite{GevKam2004} or \cite{Passenbrunner2014} and also the proofs are more or
less the same. The exception is Lemma  \ref{lem:techn2} for which we give a
renewed and shorter proof. \\

\begin{lem}\label{lem:techn1}
	Let $N(k)$ be the number given by Proposition
	\ref{prop:lower_Lp_estimate_fhat},
	$f=\sum_{n\geq N(k)}^\infty a_n \hat f_n$ and $V$ be a
 subinterval of
$\mathbb T$.
Then,
\begin{align}
\int_{V^c}\sum_{j\in\Gamma}|a_j \hat f_j(t)|\dif
t&\lesssim\int_V\Big(\sum_{j\in\Gamma}|a_j \hat f_j(t)|^2\Big)^{1/2}\dif t, \label{eq:lemtechn1:1} 
\end{align}
where
\[
\Gamma:=\{j : \hat J_j\subset  V\text{ and }N(k)\leq j<\infty\}.
\]
\end{lem}
\begin{proof}
First, assume that $|V|=1$. Then \eqref{eq:lemtechn1:1} holds trivially. In the following, we assume that $|V|<1$. 
Fixing $n\in\Gamma$, Corollary \ref{cor:geomdecay} and Proposition
\ref{prop:lower_Lp_estimate_fhat} then imply
\begin{equation}
\label{eq:techn1:1}
	\int_{V^c}|\hat f_n(t)|\dif t\lesssim \hat q^{\hat d_n(V^c)}|\hat
	J_n|^{1/2}  \lesssim \hat q^{\hat{d}_n(V^c)} \int_{\hat J_n} |\hat
	f_n(t)|\dif t.
\end{equation}
Now choose $\beta=1/4$ and let $\hat J_n^\beta$ be the unique closed interval that satisfies
\[
	|\hat J_n^\beta|=\beta|\hat J_n|\quad\text{and}\quad
	\operatorname{center}(\hat J_n^\beta)=\operatorname{center}( \hat J_n).
\]
Since $f_n$ is a polynomial of order $k$ on the interval $J_n$, we apply
{Corollary \ref{cor:remez}} to \eqref{eq:techn1:1} and estimate further
\begin{equation}\label{eq:techn1:2}
	\int_{V^c} |a_n \hat f_n(t)|\dif  t \lesssim \hat q^{\hat d_n(V^c)}
	\int_{\hat J_n^\beta} |a_n \hat f_n(t)|\dif t\leq \hat q^{\hat d_n(V^c)} 
	\int_{\hat J_n^\beta} \Big(\sum_{j\in\Gamma}|a_j\hat f_j(t)|^2\Big)^{1/2}\dif t.
\end{equation}
Define $\Gamma_s:=\{j\in\Gamma:\hat d_j(V^c)=s\}$ for $s\geq 0$. 
If $(\hat J_{n_j})_{j=1}^N$ is a decreasing sequence of characteristic intervals
with $n_j\in \Gamma_s$, we can split $(\hat J_{n_j})$ into at most two groups so that for each group 
 one endpoint of $\hat J_{n_j}$ coincides for each $j\in\Gamma$ since $\hat J_j\subset V$ for all $j\in
\Gamma$. 

So, Lemma \ref{lem:jinterval_periodic} implies that there exists a constant $F_{k}$ only depending on $k$, 
such that each point $t\in V$ belongs to at most $F_{k}$ intervals $\hat J_{j}^\beta$, $j\in\Gamma_s$.
Thus, summing over $j\in \Gamma_s$, we get from \eqref{eq:techn1:2}
\begin{equation*}
\begin{aligned}
	\sum_{j\in\Gamma_s}\int_{V^c} |a_j\hat f_j(t)|\dif t &\lesssim
	\sum_{j\in\Gamma_s}\hat q^s \int_{\hat J_{j}^\beta} 
	\Big(\sum_{\ell\in\Gamma}|a_\ell \hat f_\ell(t)|^2\Big)^{1/2} \dif t \\
&\lesssim \hat q^s\int_V\Big(\sum_{\ell \in\Gamma}|a_\ell \hat f_\ell(t)|^2\Big)^{1/2} \dif t.
\end{aligned}
\end{equation*}
Finally, we sum over $s\geq 0$ to obtain inequality \eqref{eq:lemtechn1:1}.
\end{proof}
Let $g$ be a real-valued function defined on the torus $\mathbb T$. In the
following, we denote by  $[g>\lambda]$ the  set $\{x\in \mathbb T: g(x)>\lambda\}$ for any number $\lambda>0$.
\begin{lem}\label{lem:triv}
Let  $f=\sum_{n=1}^\infty a_n \hat f_n$ with only finitely many nonzero coefficients $a_n$, $\lambda>0,\ r<1$ and 
\[
	E_\lambda=[Sf>\lambda],\quad
	B_{\lambda,r}=[\hat{\cM}\charfun_{E_\lambda}>r],
\]
where $Sf(t)^2 = \sum_{n=1}^\infty a_n^2\hat f_n(t)^2$ is the spline square function.
Then we have
\[
E_\lambda\subset B_{\lambda,r}.
\]
\end{lem}
\begin{proof}
Let $t\in E_\lambda$ be fixed. The square function $Sf=\big(\sum_{n=1}^\infty
|a_n\hat f_n|^2\big)^{1/2}$ is continuous except possibly at finitely many grid
points, where $Sf$ is at least continuous from one side.  
As a consequence, for $t\in E_\lambda$, there exists an interval $I\subset E_\lambda$ such that $t\in I$. This implies the following estimate:
\begin{align*}
(\hat\cM\charfun_{E_\lambda})(t)&=\sup_{t\ni U}|U|^{-1}\int_U \charfun_{E_\lambda}(x)\dif x \\
&=\sup_{t\ni U}\frac{|E_\lambda\cap U|}{|U|}\geq \frac{|E_\lambda\cap I|}{|I|}=\frac{|I|}{|I|}=1>r.
\end{align*}
The above inequality shows $t\in B_{\lambda,r}$, proving the lemma.
\end{proof}

\begin{lem}\label{lem:Sf}
Let $f=\sum_{n\geq N(k)} a_n \hat f_n$ with only finitely many nonzero
coefficients $a_n$, $\lambda>0$ and $r<1$, {where $N(k)$ is the number given by
	Proposition \ref{prop:lower_Lp_estimate_fhat}.}
Then we define
\[
	E_\lambda:=[Sf>\lambda],\qquad B_{\lambda,r}:= [\hat{\mathcal
	M}\charfun_{E_\lambda}>r],
\]
where $Sf(t)^2=\sum_{n\geq N(k)} a_n^2 \hat f_n(t)^2$ is the spline square
function.
If 
\[
\Lambda=\{n:\hat J_n\not\subset  B_{\lambda,r}\text{ and }N(k)\leq
n<\infty\}\qquad\text{and}\qquad g=\sum_{n\in\Lambda} a_n\hat f_n,
\]
we have
\begin{equation}\label{eq:Sfid}
\int_{E_\lambda}Sg(t)^2\dif t\lesssim_r \int_{E_\lambda^c}Sg(t)^2\dif t.
\end{equation}

\end{lem}
\begin{proof}
First, we observe that in the case $B_{\lambda,r}=\mathbb T$,
the index set $\Lambda$ is empty and thus, \eqref{eq:Sfid} holds trivially. 
So let us assume $B_{\lambda,r}\neq \mathbb T$.
Then, we  start the proof of \eqref{eq:Sfid} with an application of
 {Proposition}
\ref{prop:periodicpnorm} and Proposition \ref{prop:lower_Lp_estimate_fhat} to obtain
\[
\int_{E_\lambda} Sg(t)^2\dif t=\sum_{n\in\Lambda}\int_{E_\lambda}|a_n \hat f_n(t)|^2\dif t
\lesssim \sum_{n\in\Lambda} \int_{\hat J_n}|a_n\hat f_n(t)|^2\dif t.
\]
We split the latter expression into the parts
\[
I_1:=\sum_{n\in\Lambda} \int_{\hat J_n\cap E_\lambda^c}|a_n\hat f_n(t)|^2\dif t,\quad I_2:=
\sum_{n\in\Lambda} \int_{\hat J_n\cap E_\lambda}|a_n\hat f_n(t)|^2\dif t.
\]
For $I_1$, we clearly have
\begin{equation}\label{eq:lemSf:1}
I_1\leq \sum_{n\in\Lambda} \int_{E_\lambda^c}|a_n\hat f_n(t)|^2\dif t=\int_{E_\lambda^c} Sg(t)^2\dif t.
\end{equation}
It remains to estimate $I_2$. First we observe that by Lemma \ref{lem:triv}, $E_\lambda\subset B_{\lambda,r}$. 
Since the set $B_{\lambda,r}=[\hat{\mathcal M}\charfun_{E_\lambda}>r]$ is open
in $\mathbb T$, we decompose it into a countable collection of disjoint open subintervals $(V_j)_{j=1}^\infty$ 
of $\mathbb T$. Utilizing this decomposition, we estimate
\begin{equation}\label{eq:lemSf:2}
I_2\leq \sum_{n\in\Lambda}\sum_{j:|\hat J_n\cap V_j|>0} \int_{\hat J_n\cap
V_j}|a_n\hat f_n(t)|^2\dif t.
\end{equation}
If the indices $n$ and $j$ are such that $n\in\Lambda$ and $|\hat J_n\cap V_j|>0$,
then, by definition of $\Lambda$, $\hat J_n$ is an interval containing at least one endpoint 
$x$ of $V_j$ for which
\[
\hat\cM\charfun_{E_\lambda}(x)\leq r.
\]
This implies
\[
|E_\lambda\cap \hat J_n\cap V_j|\leq r\cdot |\hat J_n\cap V_j|\quad\text{or equivalently}\quad 
|E_\lambda^c\cap \hat J_n\cap V_j|\geq (1-r)\cdot |\hat J_n\cap V_j|.
\]
Using this inequality and that $|\hat f_n|^2$ is a polynomial of order $2k-1$ on
$\hat J_n$ allows us to use {Corollary \ref{cor:remez}}  to conclude from \eqref{eq:lemSf:2}
\begin{equation*}
\begin{aligned}
I_2&\lesssim_r \sum_{n\in\Lambda}\sum_{j:|\hat J_n\cap V_j|> 0}
\int_{E_\lambda^c\cap \hat J_n\cap V_j}|a_n\hat f_n(t)|^2\dif t \\
&\leq \sum_{n\in\Lambda} \int_{E_\lambda^c\cap \hat J_n\cap B_{\lambda,r}}|a_n\hat f_n(t)|^2\dif t \\
&\leq \sum_{n\in\Lambda} \int_{E_\lambda^c}|a_n\hat f_n(t)|^2\dif t=\int_{E_\lambda^c}Sg(t)^2\dif t.
\end{aligned}
\end{equation*}
The latter inequality combined with \eqref{eq:lemSf:1} completes the proof the lemma.
\end{proof}

\begin{lem}\label{lem:techn2}

Let $V$ be an open subinterval of $\mathbb T$  and
$f=\sum_{n} \hat a_n \hat f_n \in L^p(\mathbb T)$  for $p\in
(1,\infty)$ with $\supp f\subset V$. Then, there exists a number $R>1$ depending
only on $k$ such that 
\begin{equation}\label{eq:lemtechn2}
\sum_{n} R^{p \hat d_n(V)}|\hat a_n|^p \|\hat
f_n\|_{L^p(\widetilde{V}^c)}^p\lesssim_{p,R} \|f\|_p^p,
\end{equation}
with $\widetilde{V}$ being the interval with the same center as $V$ but with
three times the diameter.
\end{lem}
\begin{proof}
	We observe first that we can assume that $|V| \leq 1/3$, since otherwise
	$|\widetilde{V}^c|=0$ and the left hand side of \eqref{eq:lemtechn2} is
	zero.

We start by estimating 
	$|\hat a_n|$.
Depending on $n$, we partition $V$ into intervals 
$(A_{n,j})_{j=1}^{N_n}$, where except at most two intervals at the boundary of
$V$, we choose $A_{n,j}$ to be a grid point interval in the grid $\hat{\mathcal
T}_n$.
Let $I_{n,\ell} := [\sigma_{n,\ell},\sigma_{n,\ell+1}]$ 
be the $\ell$th grid point interval in $\hat{\mathcal T}_n$. Moreover, for a
grid point interval $I$ in grid $\hat{\mathcal T}_n$ and all subsets $E\subset
I$, we set $C_n(E)$ to be the
interval given by Proposition \ref{prop:convinterval} that satisfies
\begin{equation*}
	C_n(E) \supset I\cup\hat J_n
\end{equation*}
and $K_n(C_n(I))$ denotes the number of grid points from $\hat{\mathcal T}_n$
that are contained in the set $C_n(I)$.
Next, we define $r_n = \min_{\ell : I_{n,\ell}\cap \widetilde{V}^c\neq \emptyset} 
	K_n(C_n(I_{n,\ell}))$, $a_{n,j}=
	K_n(C_n(A_{n,j}))$ and we choose a number $S>1$ which we will specify
	later and 
estimate by H\"{o}lder's inequality with the dual exponent $p'=p/(p-1)$ to
$p$, 
\begin{align*}
	|\hat a_n| = |\langle f,\hat f_n\rangle| &= \Big| \sum_{j=1}^{N_n} \int_{A_{n,j}} f(t) \hat f_n(t) \dif
	t\Big|  \\
	&\leq \sum_{j=1}^{N_n} \Big( \int_{A_{n,j}}|f(t)|^p\dif t\Big)^{1/p} \Big(
	\int_{A_{n,j}} |\hat f_n(t)|^{p'} \dif t\Big)^{1/p'} \\
	&= \sum_{j=1}^{N_n}  S^{-a_{n,j}}S^{a_{n,j}}\Big( \int_{A_{n,j}}|f(t)|^p\dif t\Big)^{1/p} \Big(
	\int_{A_{n,j}} |\hat f_n(t)|^{p'} \dif t\Big)^{1/p'}\\
	&\leq \Big(\sum_{j=1}^{N_n} S^{-p'a_{n,j}}\Big)^{1/p'}
	\bigg(\sum_{j=1}^{N_n} S^{pa_{n,j}} \int_{A_{n,j}}|f(t)|^p\dif t\cdot \Big(
	\int_{A_{n,j}} |\hat f_n(t)|^{p'} \dif t\Big)^{p-1} \bigg)^{1/p}.
\end{align*}
Since the first sum above is a geometric series and by 
using Corollary \ref{cor:geomdecay} on the integral of $\hat f_n$, we obtain
\begin{equation}
	\label{eq:an}
	|\hat a_n| \lesssim 
	\bigg(\sum_{j=1}^{N_n} S^{pa_{n,j}} \int_{A_{n,j}}|f(t)|^p\dif t\cdot
	|\hat J_n|^{p/2}
	\frac{\hat q^{p a_{n,j}}|A_{n,j}|^{p-1}}{|C_n(A_{n,j})|^p}
	\bigg)^{1/p}.
\end{equation}
We also estimate $\|\hat f_n\|_{L^p(\widetilde{V}^c)}^p$ by Corollary
\ref{cor:geomdecay} and get
\begin{equation*}
	\| \hat f_n \|_{L^p(\widetilde{V}^c)}^p \lesssim |\hat J_n|^{p/2} \hat q^{p r_n}
	\sum_{\ell : \widetilde{V}^c \cap I_{n,\ell}\neq\emptyset}
	\frac{|\widetilde{V}^c\cap I_{n,\ell}|}{|C_n(I_{n,\ell})|^p} = 
|\hat J_n|^{p/2} \hat q^{p r_n}
	\int_{ \widetilde{V}^c } \sum_{\ell : \widetilde{V}^c \cap I_{n,\ell}\neq\emptyset}
	\frac{\charfun_{I_{n,\ell}}(t)}{|C_n(I_{n,\ell})|^p}\dif t.
\end{equation*}
By integration of the function $t\mapsto t^{-p}$, this is dominated by
\begin{equation} \label{eq:intermediate}
	\frac{|\hat J_n|^{p/2} \hat q^{pr_n}}{\min_{\ell:\widetilde{V}^c\cap I_{n,\ell}\neq
\emptyset}|C_n(I_{n,\ell})|^{p-1}}.
\end{equation}
For an arbitrary set $E\subset \mathbb T$,  let $\ell_0(E)$ be an index such that
$I_{n,\ell_0(E)}\cap E\neq \emptyset$ and 
\begin{equation*}
	|C_n(I_{n,\ell_0(E)})| = \min_{\ell : I_{n,\ell}\cap E\neq \emptyset}
	|C_n(I_{n,\ell})|.
\end{equation*}
Then, we introduce one more notation and set $B_n(E) \subset C_n(I_{n,\ell_0(E)})$
to be the largest interval $B$ such that $B\cap E = \hat J_n\cap E$. Obviously 
$B_n(E)\supset \hat J_n$ for every $E$. Using this
notation, we estimate \eqref{eq:intermediate} and conclude
\begin{equation}\label{eq:intfhat}
	 \|\hat f_n\|_{ L^p(\widetilde{V}^c)}^p \lesssim\frac{|\hat J_n|^{p/2}
	\hat q^{pr_n}}{|B_n(\widetilde{V}^c)|^{p-1}}.
\end{equation}

Combining \eqref{eq:an} and \eqref{eq:intfhat} yields
\begin{align*}
	&\sum_{n} R^{p\hat d_n(V)}|\hat a_n|^p \|\hat
	f_n\|_{L^p(\widetilde{V}^c)}^p \\ 
	&\lesssim
	\sum_{n} |\hat J_n|^{p}\hat q^{p r_n}R^{p\hat d_n(V)}
		|B_n(\widetilde{V}^c)|^{1-p}\cdot 
		\bigg(\sum_{j=1}^{N_n} (\hat q S)^{pa_{n,j}} \int_{A_{n,j}}|f(t)|^p\dif t\cdot 
		\frac{|A_{n,j}|^{p-1}}{|C_n(A_{n,j})|^p}
	\bigg).
\end{align*}
Since $(A_{n,j})_{j=1}^{N_n}$ is a partition of $V$ for any $n$, we further
write
\begin{align*}
	&\sum_{n} R^{p\hat d_n(V)}|\hat a_n|^p \|\hat
	f_n\|_{L^p(\widetilde{V}^c)}^p \\ 
	&\lesssim \int_V \sum_{n}
	\Big(\frac{|\hat J_n|}{|B_n(\widetilde{V}^c)|}\Big)^{p-1} \hat q^{pr_n} R^{p\hat
	d_n(V)} 
	\sum_{j=1}^{N_n}(\hat qS)^{pa_{n,j}} \frac{|\hat J_n||A_{n,j}|^{p-1}}{|C_n(A_{n,j})|^p}
\charfun_{A_{n,j}}(t) |f(t)|^p \dif t.
\end{align*}
In order to estimate this by $\int_V |f(t)|^p\dif t$, we estimate pointwise for
fixed $t\in V$. To do this, we first observe that we have to estimate the
expression
\begin{equation*}
	\sum_{n} \Big(\frac{|\hat J_n|}{|B_n(\widetilde{V}^c)|}\Big)^{p-1} \hat q^{pr_n} R^{p\hat
	d_n(V)} 
	(\hat q S)^{pa_{n,j(n)}} \frac{|\hat J_n||A_{n,j(n)}|^{p-1}}{|C_n(A_{n,j(n)})|^p},
\end{equation*}
where $A_{n,j(n)}$ is just the interval $A_{n,j}$ such that $t\in A_{n,j}$.
Next, we split the  summation index set  into $\cup T_{s}$, where
\begin{equation*}
	T_{s} = \{n : r_n+a_{n,j(n)} = s \}.
\end{equation*}
 {Since $\hat d_n(V) \leq a_{n,j(n)}$, we get that if $R,S>1$ are such that
$RS\hat q<1$, there exists $\alpha<1$
depending only on $k$, such that the above expression is $\lesssim$}
\begin{equation}
	\label{eq:starting_point}
	\sum_{s=0}^\infty \alpha^s \sum_{n\in
		T_{s}}\Big(\frac{|\hat J_n|}{|B_n(\widetilde{V}^c)|}\Big)^{p-1}
	\frac{|\hat J_n||A_{n,j(n)}|^{p-1}}{|C_n(A_{n,j(n)})|^p}.
\end{equation}
Now, we split the analysis of this expression into two cases: 

\textsc{Case 1}: $T_{s,1} = \{n\in T_s : |B_n(\widetilde{V}^c)|\leq |B_n(V)|$ or $|V|\leq
|\hat J_n|$\}:

We want to estimate the inner sum in \eqref{eq:starting_point} over $n\in
T_{s,1}$, which in the present case is
immediately estimated by 
\begin{equation*}
	\sum_{n\in T_{s,1}} 
	\frac{|\hat J_n||A_{n,j(n)}|^{p-1}}{|C_n(A_{n,j(n)})|^p }.
\end{equation*}
In order to estimate this sum, we further split the set $T_{s,1}$ into
\begin{align*}
	S_{1} &= \{ n\in T_{s,1} : \hat J_n \text{ contains at least one of the two
	endpoints of $V$}\}, \\
	S_{2} &= T_{s,1} \setminus S_{1}.
\end{align*}
By the conditions of Case 1 and the definition of $\widetilde{V}$, if $n\in
S_{1}$, we have $|\hat J_n| \geq |V|$
and  a geometric decay in the length of $\hat J_n$ by Corollary
\ref{cor:J_n decay}, therefore,
\begin{equation*}
\sum_{n\in S_{1}} 
	\frac{|\hat J_n||A_{n,j(n)}|^{p-1}}{|C_n(A_{n,j(n)})|^p }
	\leq 
	\sum_{n\in S_{1}} 
	\frac{|\hat J_n||V|^{p-1}}{|C_n(A_{n,j(n)})|^p} \leq 
	\sum_{n\in S_{1}} 
	\Big(\frac{|V|}{|\hat J_n|}\Big)^{p-1} \lesssim 1.
\end{equation*}

Next, observe that under the conditions in Case 1 and the definition of
$S_{2}$, we have $|\hat J_n\cap V|=0$ for $n\in S_2$. 
Since additionally $(A_{n,j(n)})$ is a decreasing family of subsets of $V$ and since $r_n+a_{n,j(n)}=s$
for $n\in T_s$, we can split $S_2$ into two subsets $S_{2,1}$ and $S_{2,2}$ such
that for two different indices $n_1,n_2 \in S_{2,i}$ for $i\in \{1,2\}$, we have
that the corresponding intervals $\hat J_{n_1}$ and $\hat J_{n_2}$ are either disjoint or share
an endpoint.

If $n\in S_{2,i}$ then an endpoint $a$ of $B_n(V)$ coincides with an endpoint of
$V$ (since $\hat J_n\subset V^c$). In this case, we let 
$B_n(t)\subset B_n(V)$ be the interval with the endpoints $t$ and $a$ for $t\in
B_n(V)$.
Let $\hat J_n^\beta$ for $\beta=1/4$ be the interval characterized by the properties
\begin{equation*}
	\hat J_n^\beta \subset \hat J_n, \qquad \operatorname{center}(\hat J_n^\beta) =
	\operatorname{center}(\hat J_n),\qquad  |\hat J_n^\beta| = |\hat J_n|/4.
\end{equation*}
By Lemma \ref{lem:jinterval_periodic}, for each
point $u\in\mathbb T$, there exist at most $F_k$ indices in $S_{2,i}$ such that $u\in
\hat J_{n}^\beta$. 
We now enumerate the intervals $\hat J_n$ with $n\in S_{2,i}$ in the following way:
Since those are nested, we write $\hat J_{n_{\ell,1}}$ for the maximal ones under the
inclusion relation and we enumerate as $\hat J_{n_{\ell,j}}$ such that 
$\hat J_{n_{\ell, j+1}} \subset \hat J_{n_{\ell,j}}$ for all $j$. Since an endpoint of
$\hat J_{n_2}\subset \hat J_{n_1}$ for $n_1,n_2\in S_{2,i}$ coincides, for each maximal
interval $\hat J_{n_{\ell,1}}$, we have at most two sequences of this form.

Using this enumeration, we write
\begin{align*}
	\sum_{n\in S_{2,i}} 
	\frac{|A_{n,j(n)}|^{p-1}|\hat J_n|}{|C_n(A_{n,j(n)})|^p } &\leq 
		2\beta^{-1} |V|^{p-1} \sum_{\ell,j} 
		\int_{\hat J_{n_{\ell,j}}^\beta} \frac{\dif t}{|B_{n_{\ell,j}}(t)|^p}.
\end{align*}
Observe that the function 
\begin{equation*}
	x\mapsto | \{ \ell, j, t : t\in \hat J_{n_{\ell,j}}^\beta, x= |B_{n_{\ell,j}}(t)|  \} |
\end{equation*}
is uniformly bounded by $4F_k$ for all $x\geq 0$.
Since we also have the estimate
\begin{equation*}
	|V|/2\leq |B_n(t)|,\qquad n\in S_{2,i}, t\in \hat J_n^\beta, 
\end{equation*}
we conclude
\begin{equation*}
	\beta^{-1} |V|^{p-1} \sum_{\ell,j}
		\int_{\hat J_{n_{\ell,j}}^\beta} \frac{\dif t}{|B_{n_{\ell,j}}(t)|^p}
		\leq 
		4F_k\beta^{-1}|V|^{p-1}\int_{|V|/2}^\infty \frac{\dif x}{x^p}
		\leq
		C_k
\end{equation*}
where $C_k$ is some constant only depending on $k$.
This finishes the proof in the case $n\in T_{s,1}$. 

\textsc{Case 2}: $T_{s,2} = \{n\in T_s : |B_n(V)|\leq |B_n(\widetilde{V}^c)|$
and $|\hat J_n|\leq
|V|$\}: \\
Observe that for $n\in T_{s,2}$,
we have $\hat J_n\subset \widetilde{V}$. 
Next, we subdivide $T_{s,2}$ into generations $\mathcal G_{s,\ell}$ such that
for two indices $n_1,n_2$ in the same generation, the corresponding
characteristic intervals $\hat J_{n_1}$ and $\hat J_{n_2}$ are disjoint. We observe that
from the geometric decay of characteristic intervals, we get that
$|\hat J_n|/|V|\lesssim \kappa^{\ell}$ for some $\kappa<1$ and  $n\in \mathcal G_{s,\ell}$.
Therefore, by introducing $\beta<1$ such that $\beta(p-1)<1$ we 
continue estimating 
\eqref{eq:starting_point} by  using the inequality $|V| \lesssim
	|B_n(\widetilde{V}^c)|$ for $n\in T_{s,2}$,
\begin{equation*}
\sum_{n\in
		T_{s,2}}\Big(\frac{|\hat J_n|}{|B_n(\widetilde{V}^c)|}\Big)^{p-1}
	\frac{|\hat J_n||A_{n,j(n)}|^{p-1}}{|C_n(A_{n,j(n)})|^p} \lesssim
	\sum_{\ell=0}^\infty \kappa^{\ell(1-\beta)(p-1)}
	\sum_{n\in \mathcal G_{s,\ell}} 
	\frac{|\hat J_n|^{1+\beta(p-1)}|A_{n,j(n)}|^{p-1}}{|V|^{\beta(p-1)}|C_n(A_{n,j(n)})|^p}.
\end{equation*}
We further split $\mathcal G_{s,\ell}$ into two collections $\mathcal
G_{s,\ell}^{(i)}$, where
\begin{equation*}
	\mathcal G_{s,\ell}^{(1)} = \{n\in \mathcal G_{s,\ell} :
	|C_n(A_{n,j(n)})| \geq 1-2|V| \},\qquad \mathcal G_{s,\ell}^{(2)} =
	\mathcal G_{s,\ell} \setminus \mathcal G_{s,\ell}^{(1)}.
\end{equation*}
Since we have $|V|\leq 1/3$ and the $\hat J_n$'s in the collection $\mathcal
G_{s,\ell}^{(1)}$ are disjoint, for the sum over the first collection, we immediately see that $\sum_{n\in
	\mathcal G_{s,\ell}^{(1)}}|\hat J_n||A_{n,j(n)}|^{p-1}\leq 1$, so we next
	consider 
\begin{equation}
	\label{eq:sum_geom_beta}
	\sum_{n\in \mathcal G_{s,\ell}^{(2)}} 
	\frac{|\hat J_n|^{1+\beta(p-1)}|A_{n,j(n)}|^{p-1}}{|V|^{\beta(p-1)}|C_n(A_{n,j(n)})|^p}.
\end{equation}
To analyze this expression, we define 
	$C_n'(A_{n,j(n)})$ as 
		$C_n(A_{n,j(n)})$ if $\partial C_{n}(A_{n,j(n)})\cap
		\overline{A_{n,j(n)}}\neq \emptyset$ and as the smallest interval which
		is a subset of $C_{n}(A_{n,j(n)})$ that contains $\hat J_n$ and
	$\partial V\cap \overline{A_{n,j(n)}}$ if $\partial C_{n}(A_{n,j(n)})\cap
		\overline{A_{n,j(n)}}=\emptyset$.
		The canonical case is the first one, the second case can only
		occur if $A_{n,j(n)}$ is not a grid point interval in grid $n$
		which happens only if $A_{n,j(n)}$ lies at the boundary of $V$.
With this definition,
we consider the set of different endpoints of $C_n'(A_{n,j(n)})$
intersecting $\overline{A_{n,j(n)}}$ as
\begin{equation*}
	E_{s,\ell} = \{x\in \partial C_n'(A_{n,j(n)})\cap \overline{A_{n,j(n)}} : n\in
	\mathcal G_{s,\ell}^{(2)}\},
\end{equation*}
enumerate the set $E_{s,\ell}$ by the sequence $(x_r)_{r=1}^\infty$ which by
definition is entirely contained in $\overline{V}$ and split the collection $\mathcal G_{s,\ell}^{(2)}$ according to those
different endpoints into

\begin{equation*}
	\mathcal G_{s,\ell,r}^{(2)} = \{ n\in \mathcal G_{s,\ell}^{(2)} :
	r\text{ is minimal with }x_r\in
	\partial C_n'(A_{n,j(n)})\cap \overline{A_{n,j(n)}} \}.
\end{equation*}

If we set $\Lambda_{s,\ell} = \{r : \mathcal G_{s,\ell,r}^{(2)}\neq \emptyset\}$, we can
write and estimate \eqref{eq:sum_geom_beta}  as
\begin{equation*}
	\sum_{r\in \Lambda_{s,\ell}}\sum_{n\in \mathcal G_{s,\ell,r}^{(2)}}
	\frac{|\hat J_n|^{1+\beta(p-1)}
	|A_{n,j(n)}|^{p-1}}{|V|^{\beta(p-1)}|C_n(A_{n,j(n)})|^{p}} \lesssim
	\frac{1}{|V|^{\beta(p-1)}}\sum_{r\in \Lambda_{s,\ell}} \sum_{n\in \mathcal G_{s,\ell,r}^{(2)}}
	\frac{|\hat J_n|}{|C_{n}'(A_{n,j(n)})|^{1-\beta(p-1)}}.
\end{equation*}
Since the $\hat J_n$'s in the above sum are disjoint,  $\hat J_n\subset
	\widetilde{V}$ and $x_r$ is an endpoint of
$C_{n}'(A_{n,j(n)})$ for all $n\in \mathcal G_{s,\ell,r}^{(2)}$, we can estimate
by integration:
\begin{equation*}
	\frac{1}{|V|^{\beta(p-1)}}\sum_{r\in \Lambda_{s,\ell}} \sum_{n\in \mathcal G_{s,\ell,r}^{(2)}}
	\frac{|\hat J_n|}{|C_{n}'(A_{n,j(n)})|^{1-\beta(p-1)}} \lesssim 
	\frac{1}{|V|^{\beta(p-1)}}\sum_{r\in \Lambda_{s,\ell}} \int_0^{2|V|}
	\frac{1}{t^{1-\beta(p-1)}}  \dif t \lesssim |\Lambda_{s,\ell}|.
\end{equation*}

In order to finish our estimate, we show that $|\Lambda_{s-1,\ell}| < 8 s^2+1=:N$. If we assume
the contrary,
let $(n_i)_{i=1}^{N}$ be an increasing sequence such that 
\begin{equation*}
	n_i \in \mathcal G_{s,\ell,r_{n_i}}^{(2)}
\end{equation*}
for some different values $r_{n_i}$.
Consider $F:=A_{n_{N},j(n_{N})}$ and since the $\hat J_n$'s corresponding to $n_i$ are
disjoint, one of the two connected components of $\widetilde{V}\setminus
F$ contains $(N-1)/2=4s^2$ intervals $\hat J_{n_i}$, $i=1,\ldots,N$.
Enumerate them as $\hat J_{m_1},\ldots, \hat J_{m_{(N-1)/2}}$.

Since any real sequence of length $s^2+1$ has a monotone subsequence of length
$s$, we only have the following two possibilities:
\begin{enumerate}
	\item There is a subsequence $(\ell_i)_{i=1}^{s}$ of the sequence $(m_i)$
		such that, for each $i$,
		\begin{equation*}
			\operatorname{conv}(\hat J_{\ell_i}\cup F) \subset
			\operatorname{conv}(\hat J_{\ell_{i+1}} \cup F )
		\end{equation*}
	\item There is a subsequence $(\ell_i)_{i=1}^{s}$ of the sequence $(m_i)$
		such that, for each $i$,
		\begin{equation*}
			\operatorname{conv}(\hat J_{\ell_{i+1}}\cup F) \subset
			\operatorname{conv}(\hat J_{\ell_i}
			\cup F),
		\end{equation*}
\end{enumerate}
where by $\operatorname{conv}(U)$  for $U\subset \widetilde{V}$ we mean the smallest interval contained in
$\widetilde{V}$ that contains $U$. 

We observe that $\operatorname{conv}(\hat J_{n_i} \cup F) \subset C(A_{n_i,j(n_i)})$
for all $i$ since the sequence $(A_{n_i, j(n_i)})_i$ is decreasing and therefore, in case (1),
we have $a_{\ell_i,j(\ell_i)} \geq i$ and therefore
$a_{\ell_{s},j(\ell_{s})}\geq s$ which is in conflict with the definition of
$T_{s-1,2}$.

We now recall that $r_n = \min_{r\in \mathcal
	I_n(\widetilde{V}^c)}K_n(C_n(I_{n,\ell}))$.
	We let $i(n)$ be an index such that 
	\begin{equation*}
		r_n = K_n(C_n(I_{n,i(n)})).
	\end{equation*}

In case (2), we distinguish the two cases
\begin{enumerate}[(a)]
	\item $C_{\ell_{s}}(I_{\ell_{s},i(\ell_{s})}) \supset \cup_{j=1}^{s}
		\hat J_{\ell_j}$,

	\item $C_{\ell_{s}}(I_{\ell_{s},i(\ell_s)})$ contains the set of
		points $\{x_{r_{\ell_1}}, \ldots, x_{r_{\ell_{s}}}\}$.
\end{enumerate}
If we are in case (a), we have of course $r_n\geq s$ in contradiction to the
definition of $T_{m,2}$. If we are in case (b), since the points
$x_{r_{\ell_i}}$ are all different by definition of $\mathcal
G_{s,\ell,r}^{(2)}$ and they are all (except possibly the two endpoints of $V$) part of the grid points in the grid
corresponding to the index $\ell_{s}$, we have here as well that $r_n\geq s$,
 which shows that $|\Lambda_{s-1,\ell}|\leq 8s^2+1$ and therefore, by collecting
all estimates and summing geometric series over $\ell$ and $s$,
\begin{equation*}
	\sum_{n} R^{p\hat d_n(V)}|\hat a_n|^p \| \hat
	f_n\|_{L^p(\widetilde{V}^c)}^p \lesssim \|f\|_p^p,
\end{equation*}
which finishes the proof of the lemma.
\end{proof}

\section{Proof of the Main Theorem}\label{sec:main}
In this section, we prove our main result Theorem \ref{thm:uncond}, that is
unconditionality of periodic orthonormal spline systems corresponding to 
an arbitrary admissible point sequence $(s_n)_{n\geq 1}$ in 
$L^p(\mathbb T)$ for $p\in(1,\infty)$.
\begin{proof}[Proof of Theorem \ref{thm:uncond}]
We recall the notation
\[
Sf(t)=\Big(\sum_{n\geq N(k)} |a_n\hat f_n(t)|^2\Big)^{1/2},\quad
Mf(t)=\sup_{m\geq N(k)}\Big|\sum_{n=N(k)}^m a_n \hat f_n(t)\Big|
\]
when
\[
f=\sum_{n\geq N(k)} a_n\hat f_n.
\]
Since $(\hat f_n)_{n=1}^\infty$ is a basis in $L^p(\mathbb T),\ 1\leq p<\infty,$
by Theorem \ref{thm:boundedperiodic}, for showing its unconditionality, it
suffices to show that $(\hat f_n)_{n\geq N(k)}$ is an unconditional basis sequence 
in $L^p(\mathbb T)$.
Khintchine's inequality implies that a necessary and sufficient condition for
 this is 
\begin{equation}\label{eq:maintoprove}
\|Sf\|_p \sim_p \|f\|_p,\quad f\in L^p(\mathbb T).
\end{equation}
We will prove \eqref{eq:maintoprove} for $1<p<2$ since the cases $p>2$ then follow by a duality argument.

We first prove the inequality 
\begin{equation}\label{eq:firsttoprove}
\|f\|_p\lesssim_p \|Sf\|_p.
\end{equation}
To begin with, let $f\in L^p(\mathbb T)$ with $f=\sum_{n= N(k)}^\infty a_n f_n$.
Without loss of generality, we may assume that the sequence $(a_n)_{n\geq N(k)}$
has only finitely many nonzero entries. We will prove \eqref{eq:firsttoprove} by showing the inequality $\|Mf\|_p\lesssim_p \|Sf\|_p$ and we first observe that 
\begin{equation}\label{eq:vertfkt}
\|Mf\|_p^p = p \int_0^\infty \lambda^{p-1} \psi(\lambda)\dif \lambda,
\end{equation}
with $\psi(\lambda):=[Mf>\lambda] := \{t\in\mathbb T : Mf(t)>\lambda\}.$ Next we
decompose $f$ into the two parts $\varphi_1,\varphi_2$ and estimate the corresponding distribution functions $\psi_i(\lambda):=[M\varphi_i >\lambda/2]$, $i\in\{1,2\}$, separately. We continue with the definition of the functions $\varphi_i$. 
For $\lambda>0$, we define
\begin{align*}
E_\lambda &:= [Sf>\lambda],& B_\lambda&:=[\hat\cM\charfun_{E_\lambda}>1/2], \\
\Gamma&:=\{n:\hat J_n\subset B_\lambda,N(k)\leq n<\infty\},& \Lambda&:=\Gamma^c,
\end{align*}
where we recall that $\hat J_n$ is the characteristic interval corresponding 
to the grid point $s_n$ and the function $\hat f_n$.
Then, let
\begin{align*}
\varphi_1:= \sum_{n\in\Gamma} a_n\hat f_n\qquad\text{and}\qquad
\varphi_2:=\sum_{n\in\Lambda}a_n\hat f_n.
\end{align*}
Now we estimate $\psi_1=[M\varphi_1>\lambda/2]$:
\begin{align*}
\psi_1(\lambda)
&=|\{t\in B_\lambda: M\varphi_1(t)>\lambda/2\}|+|\{t\notin B_\lambda: M\varphi_1(t)>\lambda/2\}| \\
&\leq |B_\lambda|+\frac{2}{\lambda}\int_{B_\lambda^c}M\varphi_1(t)\dif t \\
&\leq |B_\lambda|+\frac{2}{\lambda}\int_{B_\lambda^c} \sum_{n\in \Gamma}
|a_n\hat f_n(t)|\dif t.
\end{align*}
We decompose the open set $B_\lambda$ into a disjoint collection of open subintervals of
$\mathbb T$ and apply Lemma \ref{lem:techn1} to each of those intervals to conclude from the latter expression
\begin{align*}
\psi_1(\lambda) &\lesssim |B_\lambda|+\frac{1}{\lambda}\int_{B_\lambda} Sf(t)\dif t \\
&= |B_\lambda|+\frac{1}{\lambda}\int_{B_\lambda\setminus E_\lambda} Sf(t)\dif t +\frac{1}{\lambda}\int_{E_\lambda\cap B_\lambda}Sf(t)\dif t \\
&\leq |B_\lambda|+|B_\lambda\setminus E_\lambda|+\frac{1}{\lambda}\int_{E_\lambda} Sf(t)\dif t,
\end{align*}
where in the last inequality, we simply used the definition of $E_\lambda$.
Since the Hardy-Littlewood maximal function operator {$\hat{\mathcal M}$} is of weak type (1,1), $|B_\lambda|\lesssim |E_\lambda|$ and thus we obtain finally
\begin{equation}
\label{eq:main:5}
\psi_1(\lambda)\lesssim |E_\lambda|+\frac{1}{\lambda}\int_{E_\lambda} Sf(t)\dif t.
\end{equation}
We now estimate $\psi_2(\lambda)$ and obtain from Theorem
\ref{thm:hardy-littlewood-periodic} and the fact that $\hat{\mathcal M}$ is a bounded operator on $L^2[0,1]$
\begin{equation*}
\begin{aligned}
\psi_2(\lambda)&\lesssim \frac{1}{\lambda^2}\|\hat\cM \varphi_2\|_2^2\lesssim \frac{1}{\lambda^2}\|\varphi_2\|_2^2=\frac{1}{\lambda^2}\|S\varphi_2\|_2^2 \\
&=\frac{1}{\lambda^2}\Big(\int_{E_\lambda}S\varphi_2(t)^2\dif t+\int_{E_\lambda^c}S\varphi_2(t)^2\dif t\Big).
\end{aligned}
\end{equation*}
We apply Lemma \ref{lem:Sf} {to} the former expression to get
\begin{equation}\label{eq:main:3.5}
\psi_2(\lambda)\lesssim \frac{1}{\lambda^2}\int_{E_\lambda^c} S\varphi_2(t)^2\dif t
\end{equation}
Thus, combining \eqref{eq:main:5} and \eqref{eq:main:3.5},
\begin{equation*}
\begin{aligned}
\psi(\lambda)&\leq \psi_1(\lambda)+\psi_2(\lambda)\\
&\lesssim |E_\lambda|+\frac{1}{\lambda}\int_{E_\lambda}Sf(t)\dif t+\frac{1}{\lambda^2}\int_{E_\lambda^c}Sf(t)^2\dif t.
\end{aligned}
\end{equation*}
Inserting this inequality into \eqref{eq:vertfkt}, 
\begin{equation*}
\begin{aligned}
\|Mf\|_p^p
&\lesssim p\int_0^\infty \lambda^{p-1} |E_\lambda|\dif\lambda+p\int_0^\infty \lambda^{p-2}\int_{E_\lambda} Sf(t)\dif t\dif\lambda \\
&\quad+p\int_0^\infty \lambda^{p-3} \int_{E_\lambda^c}Sf(t)^2\dif t\dif\lambda \\
&=\|Sf\|_p^p+p\int_0^1 Sf(t)\int_0^{Sf(t)}\lambda^{p-2} \dif\lambda\dif t \\
&\quad + p\int_0^1 Sf(t)^2 \int_{Sf(t)}^\infty \lambda^{p-3}\dif\lambda\dif t,
\end{aligned}
\end{equation*}
and thus, since $1<p<2$, 
\[
\|Mf\|_p \lesssim_p \|Sf\|_p.
\]
So, the inequality $\|f\|_p\lesssim_p \|Sf\|_p$ is proved.

We now turn to the proof of the inequality
\begin{equation}\label{eq:mainproofsquarefunction}
\|Sf\|_p\lesssim_p \|f\|_p,\qquad 1<p<2.
\end{equation}
It is enough to show that the operator $S$ is of weak type $(p,p)$ for each exponent $p$ in the range $1<p<2$. This is because $S$ is (clearly) also of strong type $2$ and we can use the Marcinkiewicz interpolation theorem to obtain \eqref{eq:mainproofsquarefunction}.
Thus we have to show
\begin{equation}\label{eq:mainproofweaktypesquarefunction}
|[Sf>\lambda]|\lesssim_p \frac{\|f\|_p^p}{\lambda^p},\qquad f\in L^p(\mathbb T),\ \lambda>0.
\end{equation}
We fix the function $f$ and the parameter $\lambda>0$. To begin with the proof
of \eqref{eq:mainproofweaktypesquarefunction}, we define
$G_\lambda:=[\hat{\mathcal M}f>\lambda]$ for $\lambda>0$ and observe that
\begin{equation}\label{eq:weaktypeGlambda}
|G_\lambda| \lesssim_p \frac{\|f\|_p^p}{\lambda^p},
\end{equation}
since $\hat{\mathcal M}$ is of weak type $(p,p)$, and, by the Lebesgue differentiation theorem,
\begin{equation}\label{eq:lebesgue}
|f|\leq \lambda\qquad\text{a.\,e. on $G_\lambda^c$}.
\end{equation}
We decompose the open set $G_\lambda\subset[0,1]$ into a collection $(V_j)_{j=1}^\infty$ of disjoint open subintervals of $[0,1]$ and split the function $f$ into the two parts $h$ and $g$ defined by
\begin{equation*}
h:=f\cdot\charfun_{G_\lambda^c}+\sum_{j=1}^\infty T_{V_j}f,\qquad g:=f-h,
\end{equation*}
where for fixed index $j$, $T_{V_j}f$ is the projection of $f\cdot\charfun_{V_j}$ onto the space of polynomials of order $k$ on the interval $V_j$.

We treat the functions $h,g$ separately and begin with $h$. The definition of $h$ implies
\begin{equation}\label{eq:hsquared}
\|h\|_2^2=\int_{G_\lambda^c} |f(t)|^2\dif t+\sum_{j=1}^\infty \int_{V_j} (T_{V_j}f)(t)^2\dif t,
\end{equation}
since the intervals $V_j$ are disjoint. We apply the
following argument to the second summand:
by {Corollary \ref{cor:remez}},
\begin{equation*}
	\int_{V_j} (T_{V_j} f)(t)^2\dif t \sim |V_j|^{-1} \Big( \int_{V_j}
	|T_{V_j}f(t)|\dif t\Big)^2.
\end{equation*}
Since $T_{V_j}$ is a bounded operator on $L^1$ (this can be seen as a very
special instance of Shadrin's theorem, Theorem \ref{thm:shadrin}),
\begin{equation*}
	\int_{V_j} (T_{V_j} f)(t)^2\dif t \lesssim |V_j|^{-1} \Big(\int_{V_j}
	|f(t)|\dif t\Big)^2\lesssim (\hat{\mathcal M} f(x))^2 |V_j|\leq
	\lambda^2 |V_j|,
\end{equation*}
where $x$ is a boundary point of $V_j$ and the last inequality follows from the
defining property of $V_j$.
So, by using this estimate, we obtain from \eqref{eq:hsquared}
\[
\|h\|_2^2 \lesssim \lambda^{2-p}\int_{G_\lambda^c} |f(t)|^p\dif t+\lambda^2 |G_\lambda|, 
\]
and thus, in view of \eqref{eq:weaktypeGlambda},
\[
\|h\|_2^2\lesssim_p \lambda^{2-p}\|f\|_p^p.
\]
This inequality allows us to estimate
\begin{equation*}
|[Sh>\lambda/2]|\leq \frac{4}{\lambda^2}\|Sh\|_2^2 = \frac{4}{\lambda^2}\|h\|_2^2\lesssim_p \frac{\|f\|_p^p}{\lambda^p},
\end{equation*}
which concludes the proof of \eqref{eq:mainproofweaktypesquarefunction} for the part $h$.

We turn to the proof of \eqref{eq:mainproofweaktypesquarefunction} for the function $g$. Since $p<2$, we have 
\begin{equation}\label{eq:main:8}
	Sg(t)^p=\Big(\sum_{n\geq N(k)}|\langle g,f_n\rangle|^2
	f_n(t)^2\Big)^{p/2}\leq \sum_{n\geq N(k)} |\langle g,f_n\rangle|^p |f_n(t)|^p
\end{equation}
For each index $j$, we define $\widetilde{V}_j$ to be the open interval with the same center as $V_j$ but with $5$ times its length. Then, set $\widetilde{G}_\lambda:=\bigcup_{j=1}^\infty \widetilde{V}_j$ and observe that $|\widetilde{G}_\lambda|\leq 5|G_\lambda|$. We get
\begin{equation*}
|[Sg>\lambda/2]|\leq |\widetilde{G}_\lambda|+\frac{2^p}{\lambda^p}\int_{\widetilde{G}_\lambda^c}Sg(t)^p\dif t.
\end{equation*}
By \eqref{eq:weaktypeGlambda}  and \eqref{eq:main:8}, this becomes
\[
|[Sg>\lambda/2]|\lesssim_p \lambda^{-p}\Big(\|f\|_p^p+\sum_{n\geq N(k)}^\infty 
\int_{\widetilde{G}_\lambda^c}|\langle g,\hat f_n\rangle|^p|\hat f_n(t)|^p\dif t\Big).
\]
But by definition of $g$ and the fact that $T_{V_j}$ is a bounded operator on
$L^p$,
\[
\|g\|_p^p=\sum_j \int_{V_j} |f(t)-T_{V_j}f(t)|^p\dif t\lesssim_p\sum_j\int_{V_j}|f(t)|^p\lesssim \|f\|_p^p,
\]
so in order to prove the inequality $|[Sg>\lambda/2]|\leq \lambda^{-p}\|f\|_p^p$, it is enough to show the inequality
\begin{equation}\label{eq:main:finaltoprove}
\sum_{n\geq N(k)} \int_{\widetilde{G}_\lambda^c} |\langle g,\hat
f_n\rangle|^p|\hat f_n(t)|^p\dif t\lesssim \|g\|_p^p.
\end{equation}
We now let $g_j:=g\cdot \charfun_{V_j}$. The supports of $g_j$ are therefore disjoint and we have $\|g\|_p^p=\sum_{j=1}^\infty \|g_j\|_p^p$. Furthermore $g=\sum_{j=1}^\infty g_j$ with convergence in $L^p$. Thus for each $n$, we obtain
\[
\langle g,\hat f_n\rangle=\sum_{j=1}^\infty \langle g_j,\hat f_n\rangle,
\]
and it follows from the definition of $g_j$ that
\begin{equation*}
\int_{V_j} g_j(t)p(t)\dif t=0
\end{equation*}
for each polynomial $p$ on $V_j$ of order $k$. This implies that $\langle
g_j,\hat f_n\rangle=0$ for $n<\polyfun(V_j)$, where 
\[
\polyfun(V):=\min\{n:\hat\cT_n\cap V\neq\emptyset\}.
\]
Thus we obtain for all $R>1$ and for every $n$ 
\begin{equation}\label{eq:main:9}
\begin{aligned}
|\langle g,\hat f_n\rangle |^p&=\Big| \sum_{j:n\geq\polyfun(V_j)}\langle
g_j,\hat f_n\rangle\Big|^p\leq \Big(\sum_{j:n\geq
\polyfun(V_j)}R^{\hat d_n(V_j)}|\langle g_j,\hat f_n\rangle|R^{-\hat d_n(V_j)}\Big)^p \\
&\leq \Big(\sum_{j:n\geq \polyfun(V_j)}R^{p \hat d_n(V_j)}|\langle g_j,\hat
f_n\rangle|^p\Big)\Big(\sum_{j:n\geq\polyfun(V_j)}R^{-p'\hat d_n(V_j)}\Big)^{p/p'},
\end{aligned}
\end{equation}
where $p'=p/(p-1)$.
If we fix $n\geq\polyfun(V_j)$, there is at least one point of the partition
$\hat{\mathcal T}_n$ contained in $V_j$. This implies that for each fixed $s\geq 0$,
there are at most two indices $j$ such that $n\geq \polyfun(V_j)$ and $\hat d_n(V_j)=s$. Therefore, 
\[
\Big(\sum_{j:n\geq\polyfun(V_j)}R^{-p'\hat d_n(V_j)}\Big)^{p/p'}\lesssim_p 1,
\]
thus we obtain from \eqref{eq:main:9},
\begin{equation*}
|\langle g,\hat f_n\rangle |^p\lesssim_p \sum_{j:n\geq\polyfun(V_j)}R^{p\hat
d_n(V_j)}|\langle g_j,\hat f_n\rangle|^p.
\end{equation*}
Now we insert this inequality in \eqref{eq:main:finaltoprove} to get
\begin{equation*}
\begin{aligned}
\sum_{n=N(k)}^\infty&\int_{\widetilde{G}_\lambda^c}|\langle g,\hat
f_n\rangle|^p|\hat f_n(t)|^p\dif t \\
&\lesssim_p\sum_{n=N(k)}^\infty\sum_{j:n\geq\polyfun(V_j)}R^{p \hat
d_n(V_j)}|\langle g_j,\hat f_n\rangle|^p\int_{\widetilde{G}_\lambda^c} |\hat f_n(t)|^p\dif t \\
&\leq \sum_{n=N(k)}^\infty\sum_{j:n\geq\polyfun(V_j)}R^{p \hat d_n(V_j)}|\langle
g_j,\hat f_n\rangle|^p\int_{\widetilde{V}_j^c} |\hat f_n(t)|^p\dif t \\
&\leq \sum_{j=1}^\infty\sum_{n\geq\polyfun(V_j)}R^{p \hat d_n(V_j)}|\langle
g_j,\hat f_n\rangle|^p\int_{\widetilde{V}_j^c} |\hat f_n(t)|^p\dif t 
\end{aligned}
\end{equation*}
We choose $R>1$ such that we can apply Lemma \ref{lem:techn2} to obtain
\[
\sum_{n=N(k)}^\infty\int_{\widetilde{G}_\lambda^c}|\langle g,\hat
f_n\rangle|^p|\hat f_n(t)|^p\dif t \lesssim_p \sum_{j=1}^\infty \|g_j\|_p^p = \|g\|_p^p,
\]
proving \eqref{eq:main:finaltoprove} and with it the inequality
$\|Sf\|_p^p\lesssim_p \|f\|_p^p$. Thus the proof of Theorem~\ref{thm:uncond} is completed.
\end{proof}

\subsection*{Acknowledgments}
K.~Keryan was  supported by SCS RA grant 15T-1A006 and M.~Passenbrunner
	was 
supported by the FWF, project number P27723.
Part of this work was done while K.~Keryan was visiting the Department of
Analysis, J. Kepler University Linz in January 2017.

\bibliographystyle{plain}
\bibliography{periodic}

\def\cprime{$'$}
\begin{thebibliography}{10}

\bibitem{Bockarev1975}
S.~V. Bo{\v{c}}karev.
\newblock Some inequalities for {F}ranklin series.
\newblock {\em Anal. Math.}, 1(4):249--257, 1975.

\bibitem{Boehm1980}
W.~B{\"o}hm.
\newblock Inserting new knots into {B}-spline curves.
\newblock {\em Computer-Aided Design}, 12(4):199 -- 201, 1980.

\bibitem{Ciesielski1975}
Z.~Ciesielski.
\newblock Equivalence, unconditionality and convergence a.e. of the spline
  bases in {$L_{p}$} spaces.
\newblock In {\em Approximation theory ({P}apers, {VI}th {S}emester, {S}tefan
  {B}anach {I}nternat. {M}ath. {C}enter, {W}arsaw, 1975)}, volume~4 of {\em
  Banach Center Publ.}, pages 55--68. PWN, Warsaw, 1979.

\bibitem{Ciesielski2000}
Z.~Ciesielski.
\newblock Orthogonal projections onto spline spaces with arbitrary knots.
\newblock In {\em Function spaces ({P}ozna\'n, 1998)}, volume 213 of {\em
  Lecture Notes in Pure and Appl. Math.}, pages 133--140. Dekker, New York,
  2000.

\bibitem{deBoor2012}
C.~de~Boor.
\newblock On the (bi)infinite case of {S}hadrin's theorem concerning the
  {$L_\infty$}-boundedness of the {$L_2$}-spline projector.
\newblock {\em Proc. Steklov Inst. Math.}, 277(suppl. 1):73--78, 2012.

\bibitem{Demko1977}
S.~Demko.
\newblock Inverses of band matrices and local convergence of spline
  projections.
\newblock {\em SIAM J. Numer. Anal.}, 14(4):616--619, 1977.

\bibitem{DeVoreLorentz1993}
R.~A. DeVore and G.~G. Lorentz.
\newblock {\em Constructive approximation}, volume 303 of {\em Grundlehren der
  Mathematischen Wissenschaften [Fundamental Principles of Mathematical
  Sciences]}.
\newblock Springer-Verlag, Berlin, 1993.

\bibitem{Domsta1976}
J.~Domsta.
\newblock A theorem on {B}-splines. {II}. {T}he periodic case.
\newblock {\em Bull. Acad. Polon. Sci. Ser. Sci. Math. Astronom. Phys.},
  24:1077--1084, 1976.

\bibitem{GevorkyanKamont1998}
G.~G. Gevorkyan and A.~Kamont.
\newblock On general {F}ranklin systems.
\newblock {\em Dissertationes Math. (Rozprawy Mat.)}, 374:1--59, 1998.

\bibitem{GevKam2004}
G.~G. Gevorkyan and A.~Kamont.
\newblock Unconditionality of general {F}ranklin systems in {$L^p[0,1],
  1<p<\infty$}.
\newblock {\em Studia Math.}, 164(2):161--204, 2004.

\bibitem{GevorkyanSahakian2000}
G.~G. Gevorkyan and A.~A. Sahakian.
\newblock Unconditional basis property of general {F}ranklin systems.
\newblock {\em Izv. Nats. Akad. Nauk Armenii Mat.}, 35(4):7--25, 2000.

\bibitem{Keryan2005}
K.~Keryan.
\newblock Unconditionality of general periodic {F}ranklin systems in
  ${L}^p[0,1],$ $1<p<\infty$.
\newblock {\em J. Contemp. Math. Anal.}, 40(1):13--55, 2005.

\bibitem{Passenbrunner2014}
M.~Passenbrunner.
\newblock Unconditionality of orthogonal spline systems in {$L^p$}.
\newblock {\em Studia Math.}, 222(1):51--86, 2014.

\bibitem{Passenbrunner2017}
M.~Passenbrunner.
\newblock Orthogonal projectors onto spaces of periodic splines.
\newblock {\em Journal of Complexity}, 42:85--93, 2017.

\bibitem{PassenbrunnerShadrin2013}
M.~Passenbrunner and A.~Shadrin.
\newblock On almost everywhere convergence of orthogonal spline projections
  with arbitrary knots.
\newblock {\em J. Approximation Theory}, 180:77--89, 2014.

\bibitem{Shadrin2001}
A.~Shadrin.
\newblock The {$L_\infty$}-norm of the {$L_2$}-spline projector is bounded
  independently of the knot sequence: a proof of de {B}oor's conjecture.
\newblock {\em Acta Math.}, 187(1):59--137, 2001.

\end{thebibliography}

\end{document}